\documentclass[12pt]{article}
\usepackage{verbatim}
\usepackage{fancyhdr}
\usepackage{graphicx}
\usepackage{subfigure}
\usepackage{amsmath}
\usepackage{amssymb}
\usepackage{caption}
\usepackage{verbatim}
\usepackage{fancyhdr}
\usepackage{indentfirst}
\usepackage[ansinew]{inputenc}
\usepackage{color}
\usepackage{booktabs}
\usepackage{amssymb}
\usepackage{amsmath,amssymb,graphicx}
\usepackage{amsthm}
\usepackage[titletoc]{appendix}
\usepackage{subfigure}
\usepackage{color}
\usepackage{mathrsfs}
\usepackage{multirow}
\usepackage{listings}
\usepackage{epsfig}

\numberwithin{equation}{section}

\newtheorem{thm}{Theorem}[section]
\newtheorem{lmm}[thm]{Lemma}

\newtheorem{cor}[thm]{Corollary}
\theoremstyle{definition}
\newtheorem{defi}{Definition}[section]

\newtheorem{rem}[defi]{Remark}
\newtheorem{ex}{Example}[section]
\usepackage{algorithm}
\usepackage{algorithmic}
\usepackage{float}
\usepackage{lipsum}


\renewcommand{\theequation}{\arabic{section}.\arabic{equation}}

\newcommand{\nn}{\nonumber}

\allowdisplaybreaks
\begin{document}

\title{Entropy dissipative higher order accurate positivity preserving time-implicit discretizations for nonlinear degenerate parabolic equations}%
\author{Fengna Yan$^{1,2}$,\ J. J. W. Van der Vegt$^{2}$,\ Yinhua Xia$^{3}$,\ Yan Xu$^{3}$}
\date{ }
\maketitle
 {\let\thefootnote\relax\footnotetext{\, Email address: fnyan@hfut.edu.cn (F. Yan), j.j.w.vandervegt@utwente.nl (J. J. W. Van der Vegt),

 yhxia@ustc.edu.cn (Y. Xia), yxu@ustc.edu.cn (Y. Xu).}}
\footnotetext[1]{\, School of Mathematics, Hefei University of Technology, Hefei, Anhui, 230000, PR China.}

\footnotetext[2]{\, Department of Applied Mathematics, Mathematics of Computational Science Group, University

of Twente, Enschede, 7500 AE, The Netherlands.}

\footnotetext[3]{\, School of Mathematics, University of Science and Technology of China, Hefei, Anhui, 230026, PR

China.}


\begin{abstract}
We develop entropy dissipative higher order accurate local discontinuous Galerkin (LDG) discretizations coupled with
Diagonally Implicit Runge-Kutta (DIRK) methods for nonlinear degenerate parabolic equations with a gradient flow structure. Using the simple alternating numerical flux, we construct DIRK-LDG discretizations that combine the advantages of higher order accuracy, entropy dissipation and proper long-time behavior. The implicit time-discrete methods greatly alleviate the time-step restrictions needed for the stability of the numerical discretizations. Also, the larger time step significantly improves computational efficiency. We theoretically prove the unconditional entropy dissipation of the implicit Euler-LDG discretization. Next, in order to ensure the positivity of the numerical solution, we use the Karush-Kuhn-Tucker (KKT) limiter, which couples the positivity inequality constraint with higher order accurate DIRK-LDG discretizations using Lagrange multipliers. In addition,  mass conservation of the positivity-limited solution is ensured by imposing a mass conservation equality constraint to the KKT equations. The unique solvability and unconditional entropy dissipation for an implicit first order accurate in time, but higher order accurate in space, KKT-LDG discretizations are proved, which provides a first theoretical analysis of the KKT limiter. Finally, numerical results demonstrate the higher order accuracy and entropy dissipation of the KKT-DIRK-LDG discretizations for problems requiring a positivity limiter.
\end{abstract}

\textbf{Keywords:} Local discontinuous Galerkin discretizations, DIRK methods, Nonlinear degenerate parabolic equations, Unconditional entropy dissipation, KKT limiter.

\section{Introduction}

Consider the following degenerate parabolic equation \cite{bessemoulin2012finite}
\begin{align}\label{1.a}
\begin{cases}
u_t=\nabla\cdot(f(u)\nabla(\phi(\pmb{x})+H'(u))), &{\rm{ in }}\ \Omega\times(0,T],\\
u(\pmb{x},0)=u_0(\pmb{x}), &\rm{ in }\ \Omega,
\end{cases}
\end{align}
with zero-flux boundary condition
\begin{align}\label{1.abound}
\nabla(\phi(\pmb{x})+H'(u))\cdot \pmb{\nu}=0,\quad {\rm{ on }}\ \partial\Omega\times(0,T],
\end{align}
where $\Omega$ is an open bounded domain in $\mathbb{R}^d, d=1,2$, with unit outward normal vector $\pmb{\nu}$ at the boundary $\partial\Omega$, $u(\pmb{x}, t)\geqslant 0$ is a nonnegative density with time derivative denoted as $u_t$, $\phi(\pmb{x})$ is a given potential function for $\pmb{x}\in\mathbb{R}^d$, $f, H$ are given functions such that
\begin{align}\label{1.fh}
f:  \mathbb{R}^+ \xrightarrow {} \mathbb{R}^+,\quad H:  \mathbb{R}^+ \xrightarrow {} \mathbb{R},\quad f(u)H''(u)\geqslant 0,
\end{align}
where $\mathbb{R}^+$ is the nonnegative real space. Here $f(u)H''(u)$ can vanish for certain values of $u$, resulting in degenerate cases. The entropy corresponding to (\ref{1.a}) is defined by
\begin{align}\label{1.b}
E(u)=\int_\Omega (u\phi(\pmb{x})+H(u))d\Omega.
\end{align}
Multiplying (\ref{1.a}) with $\phi(\pmb{x})+H'(u)$ and integrating over $\Omega$, with the zero-flux boundary condition (\ref{1.abound}), together with (\ref{1.b}), we obtain that the time derivative of the entropy satisfies
\begin{align}\label{1.c}
\frac d{dt}E(u)=-\int_\Omega f(u)|\nabla(\phi(\pmb{x})+H'(u))|^2d\Omega\leqslant 0.
\end{align}
System (\ref{1.a}) can represent different physical problems, such as the porous media equation \cite{vazquez2007porous, zhang2009numerical}, the nonlinear nonlocal equation with a double-well potential \cite{carrillo2015finite}, the nonlinear Fokker-Plank model for fermion and boson gases \cite{abdallah2011minimization, carrillo2009fermi, toscani2012finite}.

Recently, many numerical discretizations have been proposed for (\ref{1.a}); e.g. mixed finite element methods \cite{burger2009mixed}, finite volume methods \cite{bessemoulin2012finite, carrillo2015finite}, DG methods \cite{liu2016entropy, liu2014maximum, liu2015entropy} and LDG methods \cite{zhang2009numerical}. Regarding positivity preserving discretizations, Liu and Yu developed in \cite{liu2014maximum, liu2015entropy}, respectively, for the linear Fokker-Plank equation a maximum preserving DG scheme and an entropy satisfying DG scheme, but these discretizations can not be directly applied to the general case given by (\ref{1.a}). Liu and Wang subsequently developed in \cite{liu2016entropy} an explicit Runge-Kutta (RK) time-discrete method for (\ref{1.a}) in one dimension together with a positivity preserving high order accurate DG scheme under some Courant-Friedrichs-Lewy (CFL) constraints. For the porous media equation, an LDG discretization coupled with an explicit RK method was considered in \cite{zhang2009numerical}, which is similar to the DG method in  \cite{liu2016entropy}. Still, it uses a special numerical flux to ensure the non-negativity of the numerical solution. Cheng and Shen in \cite{cheng2022} propose a Lagrange multiplier approach to construct positivity preserving schemes for a class of parabolic equations, which is different from (\ref{1.a}), but contains the porous media equation.

For the time-step $\tau$ and mesh size $h$, the condition $\tau=O(h^2)$ is needed for stability in \cite{liu2016entropy} and \cite{zhang2009numerical}. Therefore, these explicit time discretizations suffer from severe time step restrictions, but there are currently no feasible positivity preserving time-implicit LDG discretizations for (\ref{1.a}). In this paper, we present  higher order accurate Diagonally Implicit Runge-Kutta (DIRK) LDG discretizations, which ensure positivity and mass conservation of the numerical solution without the severe time step restrictions of explicit methods.

The LDG method proposed by Cockburn and Shu in \cite{cockburn1998local} has many advantages, including high parallelizability, high order accuracy, a simple choice of trial and test spaces and easy handling of complicated geometries. We refer to \cite{cockburn2001superconvergence, guo2018high, tian2016h, zhou2018stability} for examples of applications of the LDG method.

For many physical problems, it is crucial that the numerical discretization preserves the positivity properties of the partial differential equations (PDEs). Not only is this necessary to obtain physically meaningful solutions, but also negative values may result in ill-posedness of the problem and divergence of the numerical discretization. Positivity preserving DG methods have been extensively studied by many mathematicians. However, most positivity preserving DG methods are combined with explicit time-discretizations \cite{liu2016entropy, yang2013discontinuous, zhang2010maximum, zhang2010positivity}, for which numerical stability frequently imposes severe time step restrictions. These severe time-step constraints make explicit methods impractical for parabolic PDEs, such as (\ref{1.a}).

Recently, Qin and Shu extended in \cite{qin2018implicit}  the general framework for establishing positivity-preserving schemes, proposed in \cite{zhang2010maximum, zhang2010positivity}, from explicit to implicit time discretizations. They developed for one-dimensional conservation laws a positivity preserving DG method with high-order spatial accuracy combined with the first-order backward Euler implicit temporal discretization. This approach requires, however, a detailed analysis of the numerical discretization to ensure positivity and it is not straightforward to extend this approach to higher order accurate time-implicit methods.  Huang and Shen in \cite{huang2021} constructed higher order linear bound preserving implicit discretizations for the Keller-Segel and Poisson-Nernst-Planck equations. Van der Vegt, Xia and Xu proposed in \cite{van2019positivity} the KKT limiter concept to construct positivity preserving time-implicit discretizations. The KKT limiter in \cite{van2019positivity} is obtained by coupling the inequality and equality constraints imposed by the physical problem with higher order accurate DIRK-DG discretizations using Lagrange multipliers. The resulting semi-smooth nonlinear equations are solved by an efficient active set semi-smooth Newton method.

In this paper, we consider a general class of nonlinear degenerate parabolic equations given by (\ref{1.a}) and aim at developing higher order accurate entropy dissipative and positivity preserving time-implicit LDG discretizations. For the spatial discretization, we use an LDG method with simple alternating numerical fluxes, which results in entropy dissipation of the semi-discrete LDG discretization. For the temporal discretization, we consider DIRK methods, which significantly enlarge the time step for stability. The unconditional entropy dissipation of the LDG discretization combined with an implicit Euler time integration method is proved theoretically. We construct positivity preserving discretizations using the KKT limiter by imposing the positivity constraint on the numerical discretization using Lagrange multipliers. The unique solvability of the resulting positivity preserving KKT system is proved. We will also prove the unconditional entropy dissipation of the positivity preserving LDG discretization when it is combined with the backward Euler time integration method. Numerical results demonstrate the accuracy and entropy dissipation of the higher order accurate positivity preserving DIRK-LDG discretizations.

This paper is organized as follows.
In Section \ref{semi}, we present the semi-discrete LDG discretization with simple alternating numerical fluxes for the nonlinear degenerate parabolic equation stated in (\ref{1.a}) and prove that the numerical approximation is entropy dissipative. Higher order accurate DIRK-LDG discretizations, which enlarge the stable time step to a great extent, are discussed in Section \ref{fully}. The unconditional entropy dissipation of the implicit Euler LDG discretizations is proved in Section \ref{Backward Euler}. In order to ensure positivity of the numerical solution and mass conservation of the positivity limited numerical discretizations, we introduce in Section \ref{KKT} the KKT system. The higher order DIRK-LDG discretizations with positivity and mass conservation constraints are formulated in Section \ref{KKT numerical} as a KKT mixed complementarity problem.  The unique solvability and unconditional entropy dissipation of the resulting algebraic system are proved in Section \ref{well-posedness and stability for KKT system1}. In Section \ref{Numerical}, numerical results are provided to demonstrate the higher order accuracy, positivity and entropy dissipation of the positivity preserving KKT-DIRK-LDG discretizations. Concluding remarks are given in Section \ref{sec:conclusions}.


\section{Semi-discrete LDG schemes}\label{semi}

 \subsection{Definitions, Notations}

Let $\mathcal{T}_{h}$ be a shape-regular tessellation of $\Omega\subset \mathbb{R}^d$, $d=1,2$, with line or convex quadrilateral elements $K$. Given the reference element $\widehat{K}=[-1,1]^{d}$. Let ${\mathcal{Q}_{k}(\widehat{K})}$ denote the space composed of the tensor product of Legendre polynomials $\mathcal{P}_{k}(\widehat{K})$ on $[-1,1]$ of degree at most $\displaystyle k\geqslant 0$. The space ${\mathcal{Q}_{k}(K)}$ is obtained by using an isoparametric transformation from element $K$ to the reference element $\widehat{K}$. The finite element spaces $V_{h}^k$ and $\pmb{W}_{h}^k$ are defined by
\begin{align*}
&{V_{h}^k}=\{v\in L^2(\Omega): \  v|_K \in {\mathcal{Q}_{k}(K)}, \ \forall K\in \mathcal{T}_{h}\},\\
&\pmb{W}_{h}^k=\{\pmb{w}\in [L^2(\Omega)]^d: \  \pmb{w}|_K \in [{\mathcal{Q}_{k}(K)}]^{d}, \  \forall K\in \mathcal{T}_{h}\},
\end{align*}
and are allowed to have discontinuities across element interfaces. Let $e$ be an interior edge connected to the ``left" and ``right'' elements denoted, respectively, by $K_L$ and $K_R$. If $u$ is a function on $K_L$ and $K_R$, we set $u^L:=\left(u|_{K_L}\right)|_e$ and $u^R:=(u|_{K_R})|_e$ for the left and right trace of $u$ at $e$.

Note that $L^{1}(\Omega)$, $L^{2}(\Omega)$ and $L^{\infty}(\Omega)$ are standard Sobolev spaces, $\|u\|_{L^{2}(\Omega)}$ is the $L^{2}(\Omega)$-norm and $(\cdot, \cdot)_\Omega$ is the $L^{2}(\Omega)$ inner product.
 For simplicity, we denote the inner product as $(u, v):=(u, v)_{\Omega}$.

\subsection{LDG discretization in space}

For the LDG discretization of (\ref{1.a}), we first rewrite this equation as a first order system
\begin{align*}
u_t=&\nabla\cdot\pmb{q},\nn \\
\pmb{q}=&f(u)\pmb{s},\nn\\
\pmb{s}=&\nabla p,\nn\\
\displaystyle p=&\phi(\pmb{x})+H'(u).
\end{align*}
Then the LDG discretization can be readily obtained by multiplying the above equations with arbitrary test functions, integrating by parts over each element $K\in \mathcal{T}_{h}$, and finally a summation of element and face contributions. The LDG discretization can be stated as: find $u_{h}, p_{h}\in V_h^k$, $\pmb{q}_h,\pmb{s}_h\in \pmb{W}_h^k$, such that for all $\rho, \varphi \in V_h^k$ and $\pmb{\theta}, \pmb{\eta} \in \pmb{W}_h^k$, we have
\begin{subequations}\label{2.bb}
\begin{alignat}{2}\label{2.b}
(u_{ht}, \rho)+L_h^1(\pmb{q}_h; \rho)=0, \\
\label{2.b1}
(\pmb{q}_h, \pmb{\theta})+L_h^2(u_h, \pmb{s}_h; \pmb{\theta})=0,\\
\label{2.b2}
(\pmb{s}_h, \pmb{\eta})+L_h^3(p_h; \pmb{\eta})=0,\\
\label{2.b3}
(p_h, \varphi)+L_h^4(u_h; \varphi)=0,
\end{alignat}
\end{subequations}
where
\begin{subequations}\label{L1}
\begin{alignat}{2}\label{L11}
L_h^1(\pmb{q}_h; \rho):=&(\pmb{q}_h, \nabla\rho)-\sum_{K\in \mathcal{T}_{h}}(\widehat{\pmb{q}}_{h}\cdot\pmb{\nu}, \rho)_{\partial K},\\
\label{L12}
L_h^2(u_h, \pmb{s}_h; \pmb{\theta}):=&-(f(u_h)\pmb{s}_h, \pmb{\theta}),\\
\label{L13}
L_h^3(p_h; \pmb{\eta}):=&(p_h, \nabla\cdot\pmb{\eta})-\sum_{K\in \mathcal{T}_{h}}(\widehat{p}_h, \pmb{\nu}\cdot\pmb{\eta})_{\partial K},\\
\label{L14}
L_h^4(u_h; \varphi):=&-\left(\phi(\pmb{x})+H'(u_h), \varphi\right).
\end{alignat}
\end{subequations}
Note that $\pmb{\nu}$ is the unit outward normal vector of an element $K$ at its boundary $\partial K$. The ``hat" terms in $L_h^1$ and $L_h^3$ are the so-called ``numerical fluxes", whose choices play an important role in ensuring stability. We remark that the choices for the numerical fluxes are not unique. Here we use the alternating numerical fluxes
\begin{align}\label{2.a}
\widehat{\pmb{q}}_{h}=&\pmb{q}_h^R, \quad \widehat{p}_h=p_h^L,
\end{align}
or
\begin{align}\label{2.aa}
\widehat{\pmb{q}}_{h}=&\pmb{q}_h^L, \quad \widehat{p}_h=p_h^R.
\end{align}
Considering the zero-flux boundary condition $\nabla(\phi(\pmb{x})+H'(u))\cdot \pmb{\nu}=0$, we take
\begin{align}\label{2.a1}
\widehat{\pmb{q}}_{h}\cdot \pmb{\nu}=0, \quad p_h=(p_h)^{in}
\end{align}
at $\partial \Omega$, where ``in" refers to the value obtained by taking the boundary trace from the inside of the domain $\Omega$.

\subsection{Entropy dissipation}

\begin{thm}
\label{th: Stability1} For $u_{h}\in V_h^k$,  $\pmb{s}_h\in \pmb{W}_h^k$, the LDG scheme (\ref{2.bb})-(\ref{2.a1}) with $f$ satisfying (\ref{1.fh}) is entropy dissipative and satisfies
\begin{align*}
\frac{d}{dt}E(u_h)=-(f(u_h)\pmb{s}_h,\pmb{s}_h)\leqslant 0,
\end{align*}
which is consistent with the entropy dissipation property (\ref{1.c}) of the PDE (\ref{1.a}).
\end{thm}

\begin{proof}

By taking
\begin{align*}
\rho=p_h,\quad \pmb{\theta}=-\pmb{s}_h,\quad \pmb{\eta}=\pmb{q}_h,\quad \varphi=-u_{ht},
\end{align*}
in  (\ref{2.b})-(\ref{2.b3}), respectively, and after integration by parts, we have
\begin{align}\label{2.3a}
&(\phi(\pmb{x})+H'(u_h), u_{ht})\nn\\
=&-(f(u_h)\pmb{s}_h,\pmb{s}_h)-(\pmb{q}_h, \nabla p_h)+\sum_{K\in \mathcal{T}_{h}}(\widehat{\pmb{q}}_{h}\cdot\pmb{\nu}, p_h)_{\partial K}-(p_h, \nabla\cdot\pmb{q}_h)+\sum_{K\in \mathcal{T}_{h}}(\widehat{p}_h, \pmb{\nu}\cdot\pmb{q}_h)_{\partial K}\nn\\
=&-(f(u_h)\pmb{s}_h,\pmb{s}_h)-\sum_{K\in \mathcal{T}_{h}}(\pmb{q}_{h}\cdot\pmb{\nu}, p_h)_{\partial K}+\sum_{K\in \mathcal{T}_{h}}(\widehat{\pmb{q}}_{h}\cdot\pmb{\nu}, p_h)_{\partial K}+\sum_{K\in \mathcal{T}_{h}}(\widehat{p}_h, \pmb{\nu}\cdot\pmb{q}_h)_{\partial K}.
\end{align}
Assume that $e$ is an interior edge shared by elements $K_L$ and $K_R$, then $\pmb{\nu}^R=-\pmb{\nu}^L$, and together with the numerical fluxes (\ref{2.a}), we obtain
\begin{align}\label{2.3b}
&-\sum_{K_L\bigcup K_R}(\pmb{q}_{h}\cdot\pmb{\nu}, p_h)_{e}+\sum_{K_L\bigcup K_R}(\widehat{\pmb{q}}_{h}\cdot\pmb{\nu}, p_h)_{e}
+\sum_{K_L\bigcup K_R}(\widehat{p}_h, \pmb{\nu}\cdot\pmb{q}_h)_{e}\nn\\
=&-(\pmb{q}_{h}^L\cdot\pmb{\nu}^L, p_h^L)_{e}+(\pmb{q}_{h}^R\cdot\pmb{\nu}^L, p_h^R)_{e}+(\pmb{q}_{h}^R\cdot\pmb{\nu}^L, p_h^L)_{e}-(\pmb{q}_{h}^R\cdot\pmb{\nu}^L, p_h^R)_{e}\nn\\
&+(\pmb{q}_{h}^L\cdot\pmb{\nu}^L, p_h^L)_{e}-(\pmb{q}_{h}^R\cdot\pmb{\nu}^L, p_h^L)_{e}=0.
\end{align}
Combining (\ref{2.3a})-(\ref{2.3b}), using (\ref{1.b}), boundary conditions (\ref{2.a1}) and the condition on $f$ (\ref{1.fh}), we get
\begin{align*}
\frac{d}{dt}E(u_h)=(\phi(\pmb{x})+H'(u_h), u_{ht})=-(f(u_h)\pmb{s}_h,\pmb{s}_h)\leqslant 0.
\end{align*}
\end{proof}

\begin{rem}
For brevity, we will only consider in the remaining article the numerical fluxes (\ref{2.a}) and omit the discussion of the numerical fluxes (\ref{2.aa}), but all results also apply to the numerical fluxes (\ref{2.aa}).
\end{rem}

\begin{rem}
Compared to the spatial discretizations in \cite{liu2016entropy, zhang2009numerical}, we choose the simpler alternating numerical fluxes (\ref{2.a}) and (\ref{2.aa}), which significantly simplifies the theoretical analysis of the entropy dissipation property of the LDG discretization.
\end{rem}

\section{Time-implicit LDG schemes}\label{fully}

The numerical discretization of the nonlinear parabolic equations (\ref{1.a}) using explicit time discretization methods suffers from the rather severe time-step constraint $\tau=O(h^2)$. In this section, we will discuss implicit time discretizations coupled with positivity constraints in Section \ref{Positivity}.

We divide the time interval $\displaystyle[0,T]$ into $N$ parts $0=t_0<t_1<...<t_N=T$, with $\tau^n=t_n-t_{n-1}\ (n=1,2,\ldots,N)$. For $n=0,1,\ldots,N$, let $u_n=u(\cdot, t_n)$ and $u_h^n$, respectively, denote the exact and approximate values of $u$ at time $t_n$.

\subsection{Backward Euler LDG discretization}\label{Backward Euler}
Discretizing (\ref{2.bb}) in time with the implicit Euler method gives the following discrete system
\begin{subequations}\label{3.aa}
\begin{alignat}{2}\label{3.a}
\left(\frac{u_h^{n+1}-u_h^{n}}{\tau^{n+1} }, \rho\right)+L_h^1(\pmb{q}_h^{n+1}; \rho)=0, \\
\label{3.a1}
(\pmb{q}_h^{n+1}, \pmb{\theta})+L_h^2(u_h^{n+1}, \pmb{s}_h^{n+1}; \pmb{\theta})=0,\\
\label{3.a2}
(\pmb{s}_h^{n+1}, \pmb{\eta})+L_h^3(p_h^{n+1}; \pmb{\eta})=0,\\
\label{3.a3}
(p_h^{n+1}, \varphi)+L_h^4(u_h^{n+1}; \varphi)=0.
\end{alignat}
\end{subequations}
Define the discrete entropy as
 \begin{align}\label{Eh}
E_h(u_h^{n})=\int_\Omega (u_h^{n}\phi(\pmb{x})+H(u_h^{n}))dx.
\end{align}
We have the following relation for the discrete entropy dissipation.

\begin{thm}
\label{th: Stability2}
For all time levels $n$, the numerical solutions $u_h^{n}, \ u_h^{n+1}\in V_h^k$ of the LDG discretization (\ref{3.aa}), with boundary condition (\ref{2.a1}) and conditions on $f,H$ stated in (\ref{1.fh}), satisfy the following entropy dissipation relation
\begin{align}\label{3.c}
{E}_h(u_h^{n+1})\leqslant{E}_h(u_h^{n}),
\end{align}
which implies that the LDG discretization is unconditionally entropy dissipative.
\end{thm}

\begin{proof}
By choosing, respectively,  in (\ref{3.a})-(\ref{3.a3}) the following test functions
\begin{align*}
\rho=p_h^{n+1},\quad \pmb{\theta}=-\pmb{s}_h^{n+1},\quad \pmb{\eta}=\pmb{q}_h^{n+1},\quad \varphi=-\displaystyle \frac{u_h^{n+1}-u_h^{n}}{\tau^{n+1}},
\end{align*}
 we get
\begin{align*}
&\left(\phi(\pmb{x}), \displaystyle \frac{u_h^{n+1}-u_h^{n}}{\tau^{n+1}}\right)+\left(H'(u_h^{n+1}), \displaystyle \frac{u_h^{n+1}-u_h^{n}}{\tau^{n+1}}\right)\nn\\
=&-\left(f(u_h^{n+1})\pmb{s}_h^{n+1},\pmb{s}_h^{n+1}\right)-\left(\pmb{q}_h^{n+1}, \nabla p_h^{n+1}\right)+\sum_{K\in \mathcal{T}_{h}}(\widehat{\pmb{q}}_h^{n+1}\cdot\pmb{\nu}, p_h^{n+1})_{\partial K}\nn\\
&-(p_h^{n+1}, \nabla\cdot\pmb{q}_h^{n+1})+\sum_{K\in \mathcal{T}_{h}}(\widehat{p}_h^{n+1}, \pmb{\nu}\cdot\pmb{q}_h^{n+1})_{\partial K}\nn\\
=&-(f(u_h^{n+1})\pmb{s}_h^{n+1},\pmb{s}_h^{n+1})-\sum_{K\in \mathcal{T}_{h}}(\pmb{q}_h^{n+1}\cdot\pmb{\nu}, p_h^{n+1})_{\partial K}+\sum_{K\in \mathcal{T}_{h}}(\widehat{\pmb{q}}_h^{n+1}\cdot\pmb{\nu}, p_h^{n+1})_{\partial K}\nn\\
&+\sum_{K\in \mathcal{T}_{h}}(\widehat{p}_h^{n+1}, \pmb{\nu}\cdot\pmb{q}_h^{n+1})_{\partial K}.
\end{align*}
Together with (\ref{2.3b}), the numerical fluxes (\ref{2.a}) and the boundary condition (\ref{2.a1}), we obtain then
\begin{align*}
\left(\phi(\pmb{x}), \displaystyle \frac{u_h^{n+1}-u_h^{n}}{\tau^{n+1}}\right)+\left(H'(u_h^{n+1}), \displaystyle \frac{u_h^{n+1}-u_h^{n}}{\tau^{n+1}}\right)=-\left(f(u_h^{n+1})\pmb{s}_h^{n+1},\pmb{s}_h^{n+1}\right).
\end{align*}
Because of the following Taylor expansion
\begin{align*}
H(u_h^{n})=&H(u_h^{n+1})+H'(u_h^{n+1})(u_h^{n}-u_h^{n+1})+\frac12H''(\xi^{n+1})(u_h^{n+1}-u_h^{n})^2,\quad \xi^{n+1}\in(u_h^{n}, u_h^{n+1}),
\end{align*}
we have, using the conditions on $f,H$ stated in (\ref{1.fh}) and the definition of $E_h$ in (\ref{Eh}),
\begin{align*}
{E}_h(u_h^{n+1})-{E}_h(u_h^{n})=&\left(\phi(\pmb{x}), u_h^{n+1}-u_h^{n}\right)+\left(H(u_h^{n+1})-H(u_h^{n}),1\right)\nn\\
=&-\tau^{n+1}\left(f(u_h^{n+1})\pmb{s}_h^{n+1},\pmb{s}_h^{n+1}\right)-\frac12\left(H''(\xi^{n+1}), \left(u_h^{n+1}-u_h^{n}\right)^2\right)\nn\\
\leqslant&\ 0.
\end{align*}
\end{proof}

\subsection {Higher order DIRK-LDG discretizations}\label{Second}
For higher order accurate implicit in time discretizations of the system (\ref{2.bb}), we use a Diagonally Implicit Runge-Kutta (DIRK) method \cite{hairer2010solving}. Assuming we know the numerical solution at time level $n$,
we obtain the solution at time level $n+1$ with a DIRK method by solving for each DIRK stage $i, i=1,2,\cdots,s$ the following equations.
\begin{subequations}\label{3.2aa}
\begin{alignat}{2}\label{3.2a}
\left(\frac{u_h^{n+1,i}-u_h^{n}}{\tau^{n+1} }, \rho\right)+\sum_{j=1}^{i}a_{ij}L_h^1(\pmb{q}_h^{n+1,j}; \rho)=0, \\
\label{3.2a1}
(\pmb{q}_h^{n+1,i}, \pmb{\theta})+L_h^2(u_h^{n+1,i}, \pmb{s}_h^{n+1,i}; \pmb{\theta})=0,\\
\label{3.2a2}
(\pmb{s}_h^{n+1,i}, \pmb{\eta})+L_h^3(p_h^{n+1,i}; \pmb{\eta})=0,\\
\label{3.2a3}
(p_h^{n+1,i}, \varphi)+L_h^4(u_h^{n+1,i}; \varphi)=0.
\end{alignat}
 \end{subequations}
Then the solution at time $t_{n+1}$ is
\begin{align}\label{3.2c}
 (u_h^{n+1}, \rho)= &(u_h^{n}, \rho)-\tau\sum_{i=1}^{s}b_{i}L_h^1(\pmb{q}_h^{n+1,i}; \rho).
\end{align}
The coefficient matrices $(a_{ij})$ in (\ref{3.2a}) and $(b_i)$ in (\ref{3.2c}) are defined in the Butcher tableau. We choose for polynomials of order $k=1$ and  $k=2,3$ the DIRK methods introduced in \cite{alexander1977diagonally} and \cite{skvortsov2006diagonally}, respectively, that satisfy $a_{si}=b_i,\ i=1,2,\cdot\cdot\cdot,s$, which implies $u_h^{n+1}=u_h^{n+1,s}$. The order of these DIRK methods is $k+1$. The above time discretization methods are easy to implement since the matrix $(a_{ij})$ in the DIRK methods has a lower triangular structure, which means that we can compute the DIRK stages one after another, starting from $i=1$ up to $i=s$. For detailed information about the DIRK time integration method, we refer to \cite{hairer2010solving}.

\section{Higher order accurate positivity preserving DIRK-LDG discretizations}\label{Positivity}

The positivity constraints on the LDG solution will be enforced by transforming the DIRK-LDG equations with positivity constraints into a mixed complementarity problem using the Karush-Kuhn-Tucker (KKT) equations  \cite{facchinei2007finite}. In the following sections, we will first define the positivity preserving KKT-DIRK-LDG discretization. Next, we will consider the unique solvability and unconditional entropy dissipation of the discrete KKT system.
\subsection {KKT-system}\label{KKT}
For the KKT equations \cite{facchinei2007finite}, we define the set
\begin{align}\label{K}
  \mathbb{K}:=\{\widetilde{U}\in \mathbb{R}^{dof}|\ h(\widetilde{U})=0,\ g(\widetilde{U})\leqslant 0\},
 \end{align}
with equality constraints $h: \mathbb{R}^{dof} \rightarrow \mathbb{R}^{l}$ and inequality constraints $g: \mathbb{R}^{dof} \rightarrow \mathbb{R}^{m}$ being vector-valued continuously differentiable functions. The inequality constraints are used to ensure positivity. The equality constraint ensures that the limited DIRK-LDG discretization is mass conservative. Mass conservation is a property of the unlimited DIRK-LDG discretization, but one has to ensure that this property also holds after applying the positivity preserving limiter.

Let $L$ be the LDG discretization (\ref{3.2aa}) for each DIRK stage $i=1,2,\cdots,s$, without a positivity preserving limiter. We assume that $L$ is a continuously differentiable function from $ \mathbb{K}$ to $\mathbb{R}^{dof}$. The corresponding KKT-system \cite{facchinei2007finite} then is
\begin{subequations}\label{5}
\begin{alignat}{2}\label{5.1a}
L(\widetilde{U})+\nabla_{\widetilde{U}} h(\widetilde{U})^T\mu+\nabla_{\widetilde{U}} g(\widetilde{U})^T\lambda=0,\\
\label{5.1a1}
-h(\widetilde{U})=0,\\
 \label{5.1a2}
0\geqslant g(\widetilde{U})\bot \lambda\geqslant 0,
\end{alignat}
\end{subequations}
where $\mu\in \mathbb{R}^{l}$ and $\lambda\in \mathbb{R}^{m}$ are the Lagrange multipliers used to ensure $h(\widetilde{U})=0$ and $g(\widetilde{U})\leqslant 0$, respectively, $\widetilde{U}\in R^{dof}$ are the LDG coefficients in the KKT-DIRK-LDG discretization, and $\nabla_{\widetilde{U}}$ denotes the gradient with respect to $\widetilde{U}$. The compatibility condition (\ref{5.1a2}) is equivalent to
\begin{align*}
g_j(\widetilde{U})\leqslant 0,\quad \lambda_j\geqslant 0, \quad \mathrm{and} \quad g_j(\widetilde{U})\lambda_j=0,\quad j=1,2,\cdot\cdot\cdot,m,
 \end{align*}
which can be expressed as
\begin{align*}
\min(-g_j(\widetilde{U}), \lambda_j)=0,\quad j=1,2,\cdot\cdot\cdot,m.
 \end{align*}
The KKT-system then can be formulated as
 \begin{align}\label{5.1b}
 0=F(z)=\left(\begin{array}{cc}
 L(\widetilde{U})+\nabla_{\widetilde{U}} h(\widetilde{U})^T\mu+\nabla_{\widetilde{U}} g(\widetilde{U})^T\lambda \\
 -h(\widetilde{U})\\
\min(-g(\widetilde{U}), \lambda)
  \end{array}
 \right).
  \end{align}
 Here $z=(\widetilde{U},\mu,\lambda)\in\mathbb{R}^{dof+l+m}$, and $F: \mathbb{R}^{dof+l+m} \rightarrow \mathbb{R}^{dof+l+m}$ represents the DIRK-LDG discretization combined with the positivity and mass conservation constraints. Note, the KKT system (\ref{5.1b}) is nonlinear and $F(z)$ is not continuously differentiable, as is necessary for standard Newton methods, but semi-smooth. We will therefore solve (\ref{5.1b}) with the active set semi-smooth Newton method presented in \cite{van2019positivity}.

\subsection {Positivity preserving LDG discretizations}\label{KKT numerical}

In this section, we will provide the details of the higher order accurate positivity preserving DIRK-LDG discretizations (\ref{3.2aa}) coupled with the positivity and mass conservation constraints using Lagrange multipliers as stated in (\ref{5}).

 Let $N_k$ be the number of basis functions in one element. Let $N_e$ be the number of elements $K$ in the tessellation $\mathcal{T}_{h}$ of
 the domain $\Omega$. We introduce the following notation for the element-wise positivity preserving LDG solution
\begin{align*}
U_h|_K:=\sum_{j=1}^{N_k} \widetilde{U}_j^K\phi_j^K, \quad \pmb{Q}_h|_K:=\sum_{j=1}^{N_k} \widetilde{ \pmb{Q}}_j^K\phi_j^K
\end{align*}
with $K\in  \mathcal{T}_h$, $\phi_j^K$ the tensor product Legendre basis functions in $\mathcal{Q}_{k}(K)$, and LDG coefficients $\widetilde{U}_j^K\in \mathbb{R},\ \widetilde{ \pmb{Q}}_j^K\in \mathbb{R}^d$. Taking in each element $K\in \mathcal{T}_h$ the test function $\rho=\phi_j^K,\ j=1,2,\cdots, N_k$ in the operator $L_h^1(\pmb{Q}_h; \rho)$, stated in (\ref{L11}), we can define
\begin{align}\label{talta}
\mathbb{L}_h^1(\widetilde{\pmb{Q}}):=L_h^1(\pmb{Q}_h; \rho)\in \mathbb{R}^{N_kN_e},
\end{align}
with similar definitions of $\mathbb{L}_h^k$ for $L_h^k, k=2,3,4$ stated in (\ref{L12})-(\ref{L14}).

Representing the block-diagonal mass matrices for the scalar and vector variables as $M\in \mathbb{R}^{N_kN_e\times N_kN_e}$ and $\pmb{M}\in \mathbb{R}^{dN_kN_e\times dN_kN_e}$, respectively, the operator $L$ for DIRK stage $i\ (i=1,2,\cdots, s)$, as stated in (\ref{3.2a}),  can be expressed as
 \begin{align}\label{5.1c}
L(\widetilde{U}^{n+1,i}):=&M(\widetilde{U}^{n+1,i}-\widetilde{U}^{n})+\tau^{n+1}\sum_{j=1}^{i}a_{ij}\mathbb{L}_h^1(\widetilde{\pmb{Q}}^{n+1,j}),
\end{align}
with LDG  coefficients $\widetilde{U}^{n+1,i}\in\mathbb{R}^{N_kN_e}$. Similarly, using  (\ref{3.2a1}), (\ref{3.2a2}) and (\ref{3.2a3}), we have
\begin{subequations}\label{5.1cc}
\begin{alignat}{2}\label{5.1c1}
\widetilde{\pmb{Q}}^{n+1,i}=&-\pmb{M}^{-1}\mathbb{L}_h^2(\widetilde{U}^{n+1,i}, \widetilde{\pmb{S}}^{n+1,i}),\\
\label{5.1c2}
\widetilde{\pmb{S}}^{n+1,i}=&-\pmb{M}^{-1}\mathbb{L}_h^3(\widetilde{P}^{n+1,i}),\\
\label{5.1c3}
\widetilde{P}^{n+1,i}=&-M^{-1}\mathbb{L}_h^4(\widetilde{U}^{n+1,i}),
\end{alignat}
\end{subequations}
with LDG  coefficients $\widetilde{\pmb{Q}}^{n+1,i}\in \mathbb{R}^{dN_kN_e}$, $\widetilde{\pmb{S}}^{n+1,i}\in \mathbb{R}^{dN_kN_e}$, $\widetilde{P}^{n+1,i}\in \mathbb{R}^{N_kN_e}$.

The constraints on the DIRK-LDG discretization can be directly imposed on the DG coefficients for each DIRK stage using the equality and inequality constraints in the KKT-system (\ref{5.1b}). We obtain for each DIRK stage $i$, with $i=1,2,\cdots,s$, the LDG coefficients $\widetilde{U}^{n+1,i}$ by solving the following KKT system for $\widetilde{U}^{n+1,i}$,
\begin{align}\label{5.1breal}
 \left(\begin{array}{cc}
 L(\widetilde{U}^{n+1,i})+\nabla_{\widetilde{U}} h(\widetilde{U}^{n+1,i})^T\mu+\nabla_{\widetilde{U}} g(\widetilde{U}^{n+1,i})^T\lambda \\
 -h(\widetilde{U}^{n+1,i})\\
\min(-g(\widetilde{U}^{n+1,i}), \lambda)
  \end{array}
 \right)=0,
  \end{align}
where the positivity preserving inequality constraint $g(\widetilde{U}^{n+1,i})$ and the mass conservation equality constraint $h(\widetilde{U}^{n+1,i})$ are defined as follows.

 \textit{1. Positivity preserving inequality constraint}

In each element $K\in \mathcal{T}_h$, we define the function $g$ stated in (\ref{5.1breal}) as
  \begin{align}\label{5.1e}
g_{p}^K(\widetilde{U}^{n+1,i})=u_{\min}-\sum_{j=1}^{N_k}\widetilde{U}_j^{K,(n+1,i)}\phi_j^K(\pmb{x}_p), \quad p=1,\cdot\cdot\cdot,N_p,
 \end{align}
 with $N_{p}$ the number of Gauss-Lobatto quadrature points, and $\pmb{x}_p$ the Gauss-Lobatto quadrature points where the inequality constraints  $U_h(\pmb{x}_p)\geqslant u_{\min}$ are imposed. The use of Gauss-Lobatto quadrature rules ensures that the positivity constraint is also
imposed in the computation of the numerical fluxes at the element edges where
Gauss-Lobatto rules have, next to the element itself, also quadrature points. Note,
the Gauss-Lobatto quadrature points $\pmb{x}_p$ are the only points  used in  the LDG
 discretization and  the positivity constraint $u_{\min}$ therefore only needs  to be  enforced
 at these points.

 \textit{2. Mass conservation equality constraint}

 In order to ensure mass conservation of the LDG discretization when the positivity constraint is enforced, we impose the following equality constraint, which is obtained
by setting $\rho=1$ in (\ref{3.2a}) and using the numerical flux (\ref{2.a}) or (\ref{2.aa}).
 \begin{align}\label{equality constraint}
h(\widetilde{U}^{n+1,i})
=&\sum_{K\in  \mathcal{T}_h}\int_K U_h^n dK+\tau^{n+1}\sum_{j=1}^{i}a_{ij}\sum_{\substack{K\in  \mathcal{T}_h\\ \partial K\cap \partial \Omega\neq\emptyset}}(\widehat{\pmb{Q}}_h^{n+1,j}\cdot\pmb{\nu}, 1)_{\partial K}\nn\\
&-\sum_{K\in  \mathcal{T}_h}\sum_{j=1}^{N_k}\widetilde{U}_j^{K,(n+1,i)}\int_K\phi_j^K(\pmb{x}) dK,
 \end{align}
with $U_h^n$ the KKT-DIRK-LDG solution at time $t_n$.

 For each DIRK stage $i$, the KKT-system (\ref{5.1breal}) for the higher order accurate positivity preserving LDG discretization is now defined. After solving the KKT equations (\ref{5.1breal}) for $i=1,\cdots,s$, the numerical solution at time $t^{n+1}$ is directly obtained from the last DIRK stage, $U_h^{n+1}=U_h^{n+1,s}$ since we use DIRK methods with $a_{si}=b_i$.
 \begin{rem}
In order to ensure the positivity of the discrete initial solution $U_h^0$, we use the $L^2$-projection coupled with the positivity constraint  (\ref{5.1e}), which is obtained by replacing $\widetilde{U}^{n+1,i}$ with $\widetilde{U}^{0}$. The equality constraint ensures mass conservation of the positivity limited initial solution
  \begin{align*}
h(\widetilde{U}^{0})
=&\sum_{K\in  \mathcal{T}_h}\int_K u_0(\pmb{x}) dK-\sum_{K\in  \mathcal{T}_h}\sum_{j=1}^{N_k}\widetilde{U}_j^{K,0}\int_K\phi_j^K(\pmb{x}) dK.
 \end{align*}
 The constraints on the $L^2$-projection are imposed using KKT equations similar to (\ref{5.1b}). To prevent pathological cases, we assume that the limited initial solution satisfies
 \[\displaystyle \frac1{|\Omega|}\sum_{K\in  \mathcal{T}_h}\int_K u_0(\pmb{x})  dK\geqslant u_{\min}.\]
 \end{rem}
 \begin{rem}
We emphasize that $u_{\min}$ must be chosen strictly positive to ensure that errors do not violate the positivity of the numerical solution due to the finite precision of the computer arithmetic.
\end{rem}

\subsection {Unique solvability and stability of the positivity preserving LDG discretization}\label{well-posedness and stability for KKT system1}
In Section \ref{KKT numerical}, we have presented the positivity preserving LDG discretization for (\ref{1.a}). In this section, we will consider the unique solvability of the algebraic equations resulting from the backward Euler KKT-LDG discretization. In the theoretical analysis we will also consider the entropy dissipation of the positivity preserving backward Euler LDG discretization and use periodic boundary conditions.

With (\ref{5.1c})-(\ref{equality constraint}), the positivity preserving backward Euler LDG discretization results now in the following KKT system,
\begin{subequations}\label{43.1}
\begin{alignat}{2}\label{43.1a}
L(\widetilde{U}^{n+1})+\nabla_{\widetilde{U}} h(\widetilde{U}^{n+1})^T\mu^{n+1}+\nabla_{\widetilde{U}} g(\widetilde{U}^{n+1})^T\lambda^{n+1}=0,\\
\label{43.1a1}
-h(\widetilde{U}^{n+1})=0,\\
 \label{43.1a2}
\min(-g(\widetilde{U}^{n+1}), \lambda^{n+1})=0.
\end{alignat}
\end{subequations}
Here $L: \mathbb{R}^{N_kN_e} \rightarrow \mathbb{R}^{N_kN_e}$ and
 \begin{align}\label{43.1c}
L(\widetilde{U}^{n+1}):=&M\displaystyle(\widetilde{U}^{n+1}-\widetilde{U}^{n})+\tau^{n+1}B\widetilde{\pmb{Q}}^{n+1},\\
\label{43.1cQ1}
\pmb{M} \widetilde{\pmb{Q}}^{n+1}=&C_d(\widetilde{U}^{n+1})\widetilde{\pmb{S}}^{n+1},\\
\label{43.1cQ2}
\pmb{M} \widetilde{\pmb{S}}^{n+1}=&A\widetilde{P}^{n+1},\\
\label{43.1cQ3}
M \widetilde{P}^{n+1}=&D(\widetilde{U}^{n+1}).
 \end{align}
From (\ref{talta})-(\ref{5.1cc}), we obtain that
\begin{align}\label{Matrix}
B\widetilde{\pmb{Q}}^{n+1}=&\mathbb{L}_h^1(\widetilde{\pmb{Q}}^{n+1})\in \mathbb{R}^{N_kN_e},\\
\label{Matrix1}
C_d(\widetilde{U}^{n+1})\widetilde{\pmb{S}}^{n+1}=&-\mathbb{L}_h^2(\widetilde{U}^{n+1},\widetilde{\pmb{S}}^{n+1})\in \mathbb{R}^{dN_kN_e},\\
\label{Matrix2}
A\widetilde{P}^{n+1}=&-\mathbb{L}_h^3(\widetilde{P}^{n+1})\in \mathbb{R}^{dN_kN_e},\\
\label{Matrix3}
D(\widetilde{U}^{n+1})=&-\mathbb{L}_h^4(\widetilde{U}^{n+1})\in \mathbb{R}^{N_kN_e},
\end{align}
where
\begin{align}\label{Matrix4}
C_d(\widetilde{U}^{n+1})=&\left(
  \begin{array}{ccc}
    C(\widetilde{U}^{n+1}) &  & \\
      & \ddots & \\
     &  & C(\widetilde{U}^{n+1}) \\
  \end{array}
\right)
\in\mathbb{R}^{dN_kN_e\times dN_kN_e},\ C(\widetilde{U}^{n+1})\in \mathbb{R}^{N_kN_e}.
\end{align}
The constraints $h: \mathbb{R}^{N_kN_e} \rightarrow \mathbb{R},\ g: \mathbb{R}^{N_kN_e} \rightarrow \mathbb{R}^{N_pN_e}$ are defined by
 \begin{align}
\label{43.1c1}
h(\widetilde{U}^{n+1}):=&\sum_{K\in  \mathcal{T}_h}\int_K U_h^0 dK-\sum_{K\in  \mathcal{T}_h}\sum_{j=1}^{N_k}\widetilde{U}_j^{K,(n+1)}\int_K\phi_j^K(\pmb{x}) dK,\\
\label{43.1c2}
g(\widetilde{U}^{n+1}):=&(g_{1}^{K_1}(\widetilde{U}^{n+1}),\cdots,g_{N_p}^{K_1}(\widetilde{U}^{n+1}),\cdots,g_1^{K_{N_e}}(\widetilde{U}^{n+1}),\cdots,g_{N_p}^{K_{N_e}}(\widetilde{U}^{n+1})),
\end{align}
 with the definition of the constraints $g_p^{K_j},\ 1\leqslant p\leqslant N_p,1\leqslant j\leqslant N_e$ given in (\ref{5.1e}).

\subsubsection {Auxiliary results used to prove the solvability of the KKT-system}\label{Auxiliary}
In this section, we will introduce some auxiliary results, which will be used in Section \ref{well-posedness and stability for KKT system} to prove the unique solvability of the KKT-system (\ref{43.1}).

 \begin{defi}\cite[Sections 1.1, 3.2]{facchinei2007finite}\label{VI}
Let $\mathbb{K}$ be given by (\ref{K}), given a map $L: \mathbb{K}\rightarrow  \mathbb{R}^{dof}$. The Variational Inequality (VI$(\mathbb{K},L)$) is to find $\widetilde{U}\in \mathbb{K}$ such that
\begin{align}\label{VIe}
(y-\widetilde{U})^TL(\widetilde{U})\geqslant 0,\quad y\in \mathbb{K}.
\end{align}
The solution of VI$(\mathbb{K},L)$ (\ref{VIe}) is denoted by SOL$(\mathbb{K},L)$.
\end{defi}

Using the nodal basis function and the definition of $g$ in (\ref{43.1c2}) and (\ref{5.1e}), the inequality constraint set in (\ref{K}) can be written as
\begin{align}\label{Kb}
  \mathbb{K}_b:=\{\widetilde{U}\in \mathbb{R}^{dof}|\ \widetilde{U}_i^{\min}\leqslant\widetilde{U}_i\leqslant \widetilde{U}_i^{\max}, i\in\{1,\cdots,dof\}\},
 \end{align}
and we write $\mathbb{K}_b$ as
\begin{align}
  \mathbb{K}_b=\prod_{\vartheta=1}^N \mathbb{K}_{n_\vartheta},
 \end{align}
 where $\mathbb{K}_{n_\vartheta}$ is a subset of $\mathbb{R}^{n_\vartheta}$ with $\displaystyle \sum_{\vartheta=1}^N n_\vartheta=dof$. Thus for a vector $\widetilde{U}\in \mathbb{K}_b$, we write $\widetilde{U}=(\widetilde{U}_\vartheta)$, where each $\widetilde{U}_\vartheta$
belongs to $\mathbb{K}^{n_\vartheta}$.

  \begin{defi}\cite[Section 3.5.2]{facchinei2007finite}\label{uniPfunc}
Let $\mathbb{K}_b$ be given by (\ref{Kb}), a map $L: \mathbb{K}_b\rightarrow  \mathbb{R}^{dof}$ is said to be

a) a P-function on $\mathbb{K}_b$ if for all pairs of distinct vectors $\widetilde{U}$ and $\widetilde{U}'$ in $\mathbb{K}_b$,
\begin{align*}
\max_{1\leqslant \vartheta\leqslant N}(\widetilde{U}_\vartheta-\widetilde{U}'_\vartheta)^T(L_\vartheta(\widetilde{U})-L_\vartheta(\widetilde{U}'))>0,
\end{align*}

b) a uniformly P-function on $\mathbb{K}_b$ if there exists a constant $\varpi>0$ such that for all pairs of distinct vectors $\widetilde{U}$ and $\widetilde{U}'$  in $\mathbb{K}_b$,
\begin{align*}
\max_{1\leqslant \vartheta\leqslant N}(\widetilde{U}_\vartheta-\widetilde{U}'_\vartheta)^T(L_\vartheta(\widetilde{U})-L_\vartheta(\widetilde{U}'))\geqslant \varpi\|\widetilde{U}-\widetilde{U}'\|^2.
\end{align*}

\end{defi}

\begin{lmm} \cite[Proposition 3.5.10]{facchinei2007finite}\label{sym}
Let $\mathbb{K}_b$ be given by (\ref{Kb}).

a) If $L$ is a P-function on $\mathbb{K}_b$, then VI$(\mathbb{K}_b,L)$ has at most one solution.

b) If each $\mathbb{K}_{n_\vartheta}$ is closed convex and $L$ is a continuous uniformly P-function on $\mathbb{K}_b$, then the VI$(\mathbb{K}_b,L)$ has a unique solution.
\end{lmm}
\begin{lmm}  \cite[Proposition 1.3.4]{facchinei2007finite}\label{rela}
Let $\widetilde{U}\in$ SOL$(\mathbb{K},L)$ solve (\ref{VIe}) with $\mathbb{K}$ given by (\ref{K}). If Abadie's Constraint Qualification holds at $\widetilde{U}$, then there exist vectors $\mu\in\mathbb{R}^{l}$ and $\lambda\in\mathbb{R}^{m}$ satisfying the KKT system (\ref{43.1}).

Conversely, if each function $h_j\ (1\leqslant j\leqslant l)$ is affine and each function $g_i\ (1\leqslant i\leqslant m)$ is convex, and if $(\widetilde{U},\mu\, \lambda)$ satisfies (\ref{43.1}), then $\widetilde{U}$ solves VI$(\mathbb{K},L)$ given by (\ref{VIe}) with $\mathbb{K}$ given by (\ref{K}).
\end{lmm}

\subsubsection {Existence and uniqueness of LDG discretization with positivity and mass conservation constraints}\label{well-posedness and stability for KKT system}
In this section, we will prove the existence and uniqueness of the KKT system (\ref{43.1})-(\ref{43.1c2}) using the unique solvability conditions discussed in Section \ref{Auxiliary}.

\begin{lmm}\label{positive}
For periodic boundary conditions, the matrices $B$ in (\ref{Matrix}) and $A$ in (\ref{Matrix2}) satisfy $B^T=A$.
\end{lmm}
\begin{proof}
In order to prove the symmetry of $B$ in (\ref{Matrix}) and $A$ in (\ref{Matrix2}), we define the bilinear function $a:(V_{h}^{k} \times\pmb{W}_h^k)\times(V_{h}^{k} \times\pmb{W}_h^k)\rightarrow\mathbb{R}$ by
\begin{align*}
a(P_h^{n+1},\pmb{Q}_h^{n+1}; \rho,\pmb{\theta})=&(\pmb{Q}_h^{n+1}, \nabla\rho)-\sum_{K\in \mathcal{T}_{h}}(\widehat{\pmb{Q}}_h^{n+1}\cdot\pmb{\nu}, \rho)_{\partial K}\nn\\
&-(P_h^{n+1}, \nabla\cdot\pmb{\theta})+\sum_{K\in \mathcal{T}_{h}}(\widehat{P}_h^{n+1}, \pmb{\nu}\cdot\pmb{\theta})_{\partial K}.
\end{align*}
Based on the definition of  $B$ in (\ref{Matrix}) using (\ref{L11}), $A$ in (\ref{Matrix2})  using (\ref{L13}), we rewrite the above bilinear function $a$ as follows:
\begin{align*}
a(P_h^{n+1},\pmb{Q}_h^{n+1}; \rho,\pmb{\theta})
=&(\varrho,\Theta) \left(\begin{array}{cc}
0&B\\
A&0
\end{array}\right)(\widetilde{P}^{n+1},\widetilde{\pmb{Q}}^{n+1})^T,
\end{align*}
with $\varrho, \Theta$ the LDG coefficients of $\rho,\pmb{\theta}$ and $\widetilde{P}^{n+1},\widetilde{\pmb{Q}}^{n+1}$ the LDG coefficients of $P_h^{n+1},\pmb{Q}_h^{n+1}$, respectively.

Interchanging the arguments of $a$, we get
\begin{align*}
a( \rho,\pmb{\theta}; P_h^{n+1},\pmb{Q}_h^{n+1})=&(\pmb{\theta}, \nabla  P_h^{n+1})-\sum_{K\in \mathcal{T}_{h}}(\widehat{\pmb{\theta}}\cdot\pmb{\nu},  P_h^{n+1})_{\partial K}\nn\\
&-(\rho, \nabla\cdot \pmb{Q}_h^{n+1})+\sum_{K\in \mathcal{T}_{h}}(\widehat{\rho}, \pmb{\nu}\cdot \pmb{Q}_h^{n+1})_{\partial K}\nn\\
=&-( P_h^{n+1}, \nabla\cdot\pmb{\theta})+\sum_{K\in \mathcal{T}_{h}}(\pmb{\theta}\cdot\pmb{\nu},  P_h^{n+1})_{\partial K}-\sum_{K\in \mathcal{T}_{h}}(\widehat{\pmb{\theta}}\cdot\pmb{\nu},  P_h^{n+1})_{\partial K}\nn\\
&+( \pmb{Q}_h^{n+1}, \nabla\rho)-\sum_{K\in \mathcal{T}_{h}}(\rho, \pmb{\nu}\cdot \pmb{Q}_h^{n+1})_{\partial K}+\sum_{K\in \mathcal{T}_{h}}(\widehat{\rho}, \pmb{\nu}\cdot \pmb{Q}_h^{n+1})_{\partial K},
\end{align*}
Using equality (\ref{2.3b}), the alternating numerical fluxes for $\widehat{\pmb{\theta}}$ and $\widehat{\rho}$ in (\ref{2.a}) or (\ref{2.aa}), and the periodic boundary conditions, we obtain
\begin{align*}
\sum_{K\in \mathcal{T}_{h}}(\pmb{\theta}\cdot\pmb{\nu},  P_h^{n+1})_{\partial K}-\sum_{K\in \mathcal{T}_{h}}(\widehat{\pmb{\theta}}\cdot\pmb{\nu},  P_h^{n+1})_{\partial K}=&\sum_{K\in \mathcal{T}_{h}}(\widehat{P}_h^{n+1}, \pmb{\nu}\cdot\pmb{\theta})_{\partial K},\nn\\
-\sum_{K\in \mathcal{T}_{h}}(\rho, \pmb{\nu}\cdot \pmb{Q}_h^{n+1})_{\partial K}+\sum_{K\in \mathcal{T}_{h}}(\widehat{\rho}, \pmb{\nu}\cdot \pmb{Q}_h^{n+1})_{\partial K}=&-\sum_{K\in \mathcal{T}_{h}}(\widehat{\pmb{Q}}_h^{n+1}\cdot\pmb{\nu}, \rho)_{\partial K}.
\end{align*}
Hence,
\begin{align*}
a(P_h^{n+1}, \pmb{Q}_h^{n+1}; \rho,\pmb{\theta})=a( \rho,\pmb{\theta}; P_h^{n+1}, \pmb{Q}_h^{n+1}),
\end{align*}
which implies
\begin{align}\label{symm}
(\varrho,\Theta) \left(\begin{array}{cc}
0&B\\
A&0
\end{array}\right)(\widetilde{P}^{n+1},\widetilde{\pmb{Q}}^{n+1})^T=&(\widetilde{P}^{n+1},\widetilde{\pmb{Q}}^{n+1}) \left(\begin{array}{cc}
0&B\\
A&0
\end{array}\right)(\varrho,\Theta)^T\nn\\
=&(\varrho,\Theta) \left(\begin{array}{cc}
0&A^T\\
B^T&0
\end{array}\right)(\widetilde{P}^{n+1},\widetilde{\pmb{Q}}^{n+1})^T.
\end{align}
Since $(P_h^{n+1}, \pmb{Q}_h^{n+1})\in V_{h}^{k} \times\pmb{W}_h^k$ and $(\rho,\pmb{\theta})\in V_{h}^{k} \times\pmb{W}_h^k$ are arbitrary functions, relation (\ref{symm}) implies that $A=B^T$.
\end{proof}
Using (\ref{43.1cQ1})-(\ref{43.1cQ3}) and Lemma \ref{positive}, the operator $L(\widetilde{U}^{n+1})$ in (\ref{43.1c}) can be written as
\begin{align}\label{L}
L(\widetilde{U}^{n+1})=M(\widetilde{U}^{n+1}-\widetilde{U}^{n})
+\tau^{n+1}B\pmb{M}^{-1}C_d(\widetilde{U}^{n+1})\pmb{M}^{-1}B^TM^{-1}D(\widetilde{U}^{n+1}).
\end{align}

\begin{lmm}\label{exist1}
Given $\widetilde{U}^{n}$, the operator $L$ in (\ref{L}) is a uniformly P-function on $\mathbb{K}_b$.
\end{lmm}
\begin{proof}
Using relation (\ref{L}) for $L$, for arbitrary $\widetilde{U}_I^{n+1}, \widetilde{U}_{II}^{n+1}\in \mathbb{K}_b$, there holds
\begin{align}\label{P-function}
L(\widetilde{U}_I^{n+1})-L(\widetilde{U}_{II}^{n+1})
=&M(\widetilde{U}_I^{n+1}-\widetilde{U}_{II}^{n+1})
+\tau^{n+1}B\pmb{M}^{-1}C_d(\widetilde{U}_I^{n+1})\pmb{M}^{-1}B^TM^{-1}D(\widetilde{U}_I^{n+1})\nn\\
&-\tau^{n+1}B\pmb{M}^{-1}C_d(\widetilde{U}_{II}^{n+1})\pmb{M}^{-1}B^TM^{-1}D(\widetilde{U}_{II}^{n+1}).
 \end{align}
After subtracting and adding $\tau^{n+1}B\pmb{M}^{-1}C_d(\widetilde{U}_I^{n+1})\pmb{M}^{-1}B^TM^{-1}D(\widetilde{U}_{II}^{n+1})$ in (\ref{P-function}), we obtain
\begin{align}\label{P-function1}
&L(\widetilde{U}_I^{n+1})-L(\widetilde{U}_{II}^{n+1})\nn\\
=&M(\widetilde{U}_I^{n+1}-\widetilde{U}_{II}^{n+1})
+\tau^{n+1}B\pmb{M}^{-1}C_d(\widetilde{U}_I^{n+1})\pmb{M}^{-1}B^TM^{-1}(D(\widetilde{U}_I^{n+1})-D(\widetilde{U}_{II}^{n+1}))\nn\\
&+\tau^{n+1}B\pmb{M}^{-1}(C_d(\widetilde{U}_I^{n+1})-C_d(\widetilde{U}_{II}^{n+1}))\pmb{M}^{-1}B^TM^{-1}D(\widetilde{U}_{II}^{n+1}).
 \end{align}
With the definition of $D$ in (\ref{Matrix3}) using (\ref{L14}), we obtain that
\begin{align*}
&(D(\widetilde{U}_I^{n+1})-D(\widetilde{U}_{II}^{n+1}))_i
=\int_\Omega\left(H'\left(\sum_{j=1}^{N_kN_e}\widetilde{U}_{I,j}^{n+1}\phi_j\right)
-H'\left(\sum_{j=1}^{N_kN_e}\widetilde{U}_{II,j}^{n+1}\phi_j\right)\right)\phi_id\Omega\nn\\
=&\sum_{j=1}^{N_kN_e}(\widetilde{U}_{I,j}^{n+1}-\widetilde{U}_{II,j}^{n+1})
\int_\Omega H''(\xi_1^{n+1})\phi_j\phi_id\Omega, \ i\in \{1,\cdots,N_kN_e\}, \xi_1^{n+1}\in(U_{h,I}^{n+1},U_{h,II}^{n+1}),
 \end{align*}
 and write
\begin{align}\label{Du}
D(\widetilde{U}_I^{n+1})-D(\widetilde{U}_{II}^{n+1}):=&D_ {\widetilde{U}}(\xi_1^{n+1})(\widetilde{U}_I^{n+1})-\widetilde{U}_{II}^{n+1}).
 \end{align}
Similarly, from the definition of $C_d$ in (\ref{Matrix1}), (\ref{Matrix4})  using (\ref{L12}), we obtain that
\begin{align*}
&C_d(\widetilde{U}_I^{n+1})-C_d(\widetilde{U}_{II}^{n+1})=\left(
  \begin{array}{ccc}
    C(\widetilde{U}_I^{n+1})-C(\widetilde{U}_{II}^{n+1}) &  & \\
      & \ddots & \\
     &  & C(\widetilde{U}_I^{n+1})-C(\widetilde{U}_{II}^{n+1}) \\
  \end{array}
\right),\\ \\
&(C(\widetilde{U}_I^{n+1})-C(\widetilde{U}_{II}^{n+1}))_{ij}=\int_\Omega\left(f\left(\sum_{k=1}^{N_kN_e}\widetilde{U}_{I,k}^{n+1}\phi_k\right)
-f\left(\sum_{k=1}^{N_kN_e}\widetilde{U}_{II,k}^{n+1}\phi_k\right)\right)\phi_j\phi_id\Omega\nn\\
=&\sum_{k=1}^{N_kN_e}(\widetilde{U}_{I,k}^{n+1}-\widetilde{U}_{II,k}^{n+1})
\int_\Omega f'(\xi_2^{n+1})\phi_k\phi_j\phi_id\Omega,\ i,j,k\in \{1,\cdots,N_kN_e\},  \xi_2^{n+1}\in(U_{h,I}^{n+1},U_{h,II}^{n+1}),
 \end{align*}
 and write
 \begin{align}\label{Cdu}
C(\widetilde{U}_I^{n+1})-C(\widetilde{U}_{II}^{n+1})
:=\sum_{k=1}^{N_kN_e}[C_{d\widetilde{U}}(\xi_2^{n+1})]_k(\widetilde{U}_{I,k}^{n+1})-\widetilde{U}_{II,k}^{n+1}).
 \end{align}
Assume for arbitrary $\widetilde{U}\in \mathbb{K}_b$ in (\ref{Kb}), that
 \begin{align}\label{asu}
  |C(\widetilde{U})_{ij}|\leqslant& c,\quad |D(\widetilde{U})_{i}|\leqslant c,\nn\\
 |[C_ {\widetilde{U}}(\widetilde{U})_{ij}]_{k}|\leqslant& c,\quad  |D_ {\widetilde{U}}(\widetilde{U})_{ij}|\leqslant c,\ i,j,k\in \{1,\cdots,N_kN_e\},
 \end{align}
 with $c$ a positive constant, independent of $\widetilde{U}$. In the remainder of this section, $c$ is a positive constant, but not necessarily the same.

Using (\ref{Du})-(\ref{Cdu}) and assumption (\ref{asu}), we obtain the following two estimates
\begin{align*}
&(\widetilde{U}_I^{n+1}-\widetilde{U}_{II}^{n+1})^TB\pmb{M}^{-1}C_d(\widetilde{U}_I^{n+1})\pmb{M}^{-1}B^TM^{-1}(D(\widetilde{U}_I^{n+1})-D(\widetilde{U}_{II}^{n+1}))\nn\\
 \leqslant&\|B\|\|\pmb{M}^{-1}\|\|C_d(\widetilde{U}_I^{n+1})\|\|\pmb{M}^{-1}\|\|B^T\|\|M^{-1}\|\|D_{\widetilde{U}}(\xi_1^{n+1})\|\|\widetilde{U}_I^{n+1}-\widetilde{U}_{II}^{n+1}\|^2\nn\\
 \leqslant& c\|\widetilde{U}_I^{n+1}-\widetilde{U}_{II}^{n+1}\|^2,
 \end{align*}
 and
 \begin{align*}
 &(\widetilde{U}_I^{n+1}-\widetilde{U}_{II}^{n+1})^TB\pmb{M}^{-1}
 (C_d(\widetilde{U}_I^{n+1})-C_d(\widetilde{U}_{II}^{n+1}))\pmb{M}^{-1}B^TM^{-1}D(\widetilde{U}_{II}^{n+1})\nn\\
 \leqslant&\|B\|\|\pmb{M}^{-1}\|\sum_{k=1}^{N_kN_e}\|[C_{d\widetilde{U}}(\xi_2^{n+1})]_k\|\|\pmb{M}^{-1}\|
 \|B^T\|\|M^{-1}\|\|D(\widetilde{U}_{II}^{n+1})\|\|\widetilde{U}_I^{n+1}-\widetilde{U}_{II}^{n+1}\|^2\nn\\
 \leqslant& c\|\widetilde{U}_I^{n+1}-\widetilde{U}_{II}^{n+1}\|^2.
 \end{align*}
Then multiplying (\ref{P-function1}) with $(\widetilde{U}_I^{n+1}-\widetilde{U}_{II}^{n+1})^T$ gives
\begin{align}\label{P-function2}
&(\widetilde{U}_I^{n+1}-\widetilde{U}_{II}^{n+1})^T(L(\widetilde{U}_I^{n+1})-L(\widetilde{U}_{II}^{n+1}))
=(\widetilde{U}_I^{n+1}-\widetilde{U}_{II}^{n+1})^TM(\widetilde{U}_I^{n+1}-\widetilde{U}_{II}^{n+1})\nn\\
&+\tau^{n+1}(\widetilde{U}_I^{n+1}-\widetilde{U}_{II}^{n+1})^TB\pmb{M}^{-1}C_d(\widetilde{U}_I^{n+1})\pmb{M}^{-1}B^TM^{-1}(D(\widetilde{U}_I^{n+1})-D(\widetilde{U}_{II}^{n+1}))\nn\\
&+\tau^{n+1}(\widetilde{U}_I^{n+1}-\widetilde{U}_{II}^{n+1})^TB\pmb{M}^{-1}(C_d(\widetilde{U}_I^{n+1})-C_d(\widetilde{U}_{II}^{n+1}))\pmb{M}^{-1}B^TM^{-1}D(\widetilde{U}_{II}^{n+1})\nn\\
\geqslant&\sigma\|\widetilde{U}_I^{n+1}-\widetilde{U}_{II}^{n+1}\|^2
-2c\tau^{n+1}\|\widetilde{U}_I^{n+1}-\widetilde{U}_{II}^{n+1}\|^2,
 \end{align}
where $\sigma>0$ is the smallest eigenvalue of the symmetric positive mass matrix $M$.

Choosing $0<\tau^{n+1}\leqslant \displaystyle\frac{\sigma}{4c}$, we obtain that
\begin{align}\label{P-function3}
&(\widetilde{U}_I^{n+1}-\widetilde{U}_I^{n+1})^T(L(\widetilde{U}_I^{n+1})-L(\widetilde{U}_{II}^{n+1}))
\geqslant\frac{\sigma}2\|\widetilde{U}_I^{n+1}-\widetilde{U}_{II}^{n+1}\|^2,\ \forall \widetilde{U}_I^{n+1}, \widetilde{U}_{II}^{n+1}\in \mathbb{K}_b,
 \end{align}
which implies that for $\tau^{n+1}$ sufficiently small $L(\widetilde{U}^{n+1})$ is a uniformly function of $\mathbb{K}_b$,
\end{proof}

From Lemmas \ref{sym}, Lemma \ref{rela} and Lemma \ref{exist1}, we obtain the main result of this section.
\begin{thm}
\label{th: well-defined}
Given the DG coefficients $\widetilde{U}^{n}$ and the positivity preserving backward Euler KKT-LDG discretization (\ref{43.1})-(\ref{43.1c2}) with equality constraint $h\equiv 0$. If assumption (\ref{asu}) is satisfied, then the KKT system (\ref{43.1})-(\ref{43.1c2}) has only one solution.
\end{thm}
\begin{cor}
\label{th: well-defined1}
Given the DG coefficients $\widetilde{U}^{n}$.  If assumption (\ref{asu}) is satisfied, then for the degenerate parabolic equation (\ref{1.a}) with periodic boundary conditions there exists only one solution satisfying the higher order accurate in time, positivity preserving KKT-DIRK-LDG discretizations  (\ref{5.1breal}) with equality constraint $h\equiv 0$.
\end{cor}
\begin{proof}
Since the DIRK coefficient matrix $(a_{ij})$ introduced in Section \ref{Second} is a lower triangular matrix, the structure of the DIRK-LDG discretizations is similar to the form obtained for the backward Euler LDG discretization. The analysis therefore is completely analogous to Theorem \ref{th: well-defined}.
\end{proof}

\subsubsection {Stability of the KKT-LDG discretization}\label{Entropy dissipation for KKT system}
\begin{thm}
\label{th: Stability3}
Given the  numerical solution $U_h^{n}\in V_h^k$ of the positivity preserving backward Euler KKT-LDG discretization (\ref{43.1})-(\ref{43.1c2}). If assumption (\ref{asu}) is satisfied, then the discrete entropy $E_h$ stated in (\ref{Eh}) satisfies for $n=0,1,\cdots$,
\begin{align}\label{43.c}
{E}_h(U_h^{n+1})\leqslant{E}_h(U_h^{n}),
\end{align}
which implies that the positivity preserving backward Euler KKT-LDG discretization is unconditionally entropy dissipative.
\end{thm}
\begin{proof}
From Lemma \ref{rela}, we obtain that the LDG coefficients $\widetilde{U}^{n+1}$ of the positivity preserving solution $U_h^{n+1}$ solve
\begin{align}\label{VIe1}
(y-\widetilde{U}^{n+1})^TL(\widetilde{U}^{n+1})\geqslant 0,\quad \forall y\in \mathbb{K},
\end{align}
with $L$ given by (\ref{L}) and $\mathbb{K}$ given by (\ref{K}).

From assumption (\ref{asu}), we have that there exists a positive constant $c\geqslant c_0>0$ such that
\begin{align}
\widetilde{U}^{n+1}-cM^{-1}D(\widetilde{U}^{n+1})\in \mathbb{K}.
\end{align}
Next, we choose $y=\widetilde{U}^{n+1}-cM^{-1}D(\widetilde{U}^{n+1})$ in (\ref{VIe1}), which implies
\begin{align}\label{11}
-c(M^{-1}D(\widetilde{U}^{n+1}))^TL(\widetilde{U}^{n+1})\geqslant 0.
\end{align}
Using (\ref{L}) and the fact that $c>0$, we obtain that (\ref{11}) implies the inequality
\begin{align}\label{qqq}
&D(\widetilde{U}^{n+1})^T(\widetilde{U}^{n+1}-\widetilde{U}^{n})\nn\\
+&\tau^{n+1}D(\widetilde{U}^{n+1})^TM^{-1}B\pmb{M}^{-1}C_d(\widetilde{U}^{n+1})\pmb{M}^{-1}B^TM^{-1}D(\widetilde{U}^{n+1})\leqslant 0.
\end{align}
From the definition of $C_d$ in (\ref{Matrix1}), (\ref{Matrix4}) using (\ref{L12}) and the conditions on $f$ stated in (\ref{1.fh}), we obtain that $C_d(\widetilde{U}^{n+1})$ is symmetric positive definite. Hence using $\tau^{n+1}>0$, we have
\begin{align*}
\tau^{n+1}D(\widetilde{U}^{n+1})^TM^{-1}B\pmb{M}^{-1}C_d(\widetilde{U}^{n+1})\pmb{M}^{-1}B^TM^{-1}D(\widetilde{U}^{n+1})\geqslant 0,
\end{align*}
which with (\ref{qqq}) yields
\begin{align}\label{4331}
D(\widetilde{U}^{n+1})^T(\widetilde{U}^{n+1}-\widetilde{U}^{n})\leqslant0.
\end{align}
From the definition of $D$ in (\ref{Matrix3}) using (\ref{L14}) and (\ref{4331}), we obtain the bound
\begin{align}\label{22}
\left(\phi(\pmb{x}), \displaystyle U_h^{n+1}-U_h^{n}\right)+\left(H'(U_h^{n+1}), U_h^{n+1}-U_h^{n}\right)\leqslant0.
\end{align}
Using the following Taylor expansion
\begin{align*}
H(U_h^{n})=&H(U_h^{n+1})+H'(U_h^{n+1})(U_h^{n}-U_h^{n+1})
\nn\\
&+\frac12H''(\xi_3^{n+1})(U_h^{n+1}-U_h^{n})^2,\ \xi_3^{n+1}\in(U_h^{n}, U_h^{n+1}),
\end{align*}
we obtain that (\ref{22}) gives
\begin{align*}
\left(\phi(\pmb{x}), U_h^{n+1}-U_h^{n}\right)+\left(H(U_h^{n+1})-H(U_h^{n}),1\right)
+\frac12\left(H''(\xi_3^{n+1}), \left(U_h^{n+1}-U_h^{n}\right)^2\right)
\leqslant 0,
\end{align*}
which implies, using the definition of $E_h$ in (\ref{Eh}), that
\begin{align*}
{E}_h(U_h^{n+1})-{E}_h(U_h^{n})=&\left(\phi(\pmb{x}), U_h^{n+1}-U_h^{n}\right)+\left(H(U_h^{n+1})-H(U_h^{n}),1\right)
\leqslant 0,
\end{align*}
since (\ref{1.fh}) gives $H''(\xi_3^{n+1})\geqslant 0$. This proves (\ref{43.c}).
\end{proof}
 \section{Numerical tests} \label{Numerical}

 In this section, we will discuss several numerical experiments to demonstrate the performance of the KKT-DIRK-LDG positivity preserving algorithm for the degenerate parabolic equation (\ref{1.a}).
In the computations, we will consider the porous medium equation, the nonlinear diffusion equation with a double-well potential and the nonlinear Fokker-Plank equation for fermion and boson gases. Firstly, we will present in Section \ref{order} the order of accuracy of the DIRK-LDG discretizations with and without positivity preserving limiter to investigate if the limiter negatively affects the accuracy of the discretizations. Next, we will present in Sections \ref{double-well}-\ref{boson} test cases for which the positivity preserving limiter is essential. Without the positivity constraint, obtaining a numerical solution or only for extremely small time steps is impossible.

In the computations, we take $\tau=\alpha \cdot h$. If the Newton method during strongly nonlinear stages requires a large number of iterations, it is generally more efficient to reduce the time step to $\displaystyle\frac{1}2\tau$ and restart the Newton iterations. When the Newton method converges well,  then $\tau$ is increased each time step to $1.2\tau$, till the maximum predefined time step is obtained.

 In order to avoid round-off effects, a positivity bound $u_{\min}=10^{-10}$ is used in the numerical simulations, except for Section \ref{order} where $u_{\min}=10^{-14}$. If it is not stated otherwise, the numerical results for 1D problems are obtained on a mesh containing 100 elements and Legendre polynomials of order 2. For 2D problems, a mesh consisting of $30\times 30$ square elements and tensor product Legendre polynomial basis functions of order 2 are used.

\subsection {Accuracy tests}\label{order}
For the accuracy test, we use a uniform mesh with $M$ elements and positivity bound $u_{\min}=10^{-14}$.

\begin{ex}\label{order2}
We consider (\ref{1.a}) on the domain $\Omega=[-1,1]$ with Dirichlet boundary conditions based on the exact solution and select the following parameters
\begin{align*}
f(u)=u,\quad H'(u)=u^{2},\quad \phi(x)=0,\quad x\in \Omega.
\end{align*}
 Then (\ref{1.a}) with a properly chosen source term has the nonnegative solution
 \begin{align*}
u(x,t)=\exp(-t)(1-x^4)^{5},\quad x\in \Omega.
 \end{align*}

We take $\alpha$ in the definition of the time step as $\alpha=1$. Tables \ref{num2}-\ref{num3} show that the DIRK-LDG discretizations with and without positivity preserving limiter are convergent at the rate $O(h^{k+1})$ for basis functions with polynomial order ranging from 1 to 3. The errors and orders of accuracy presented in Tables \ref{num2}-\ref{num3} indicate that the positivity preserving limiter is necessary and does not negatively affect accuracy.

 {

 \begin{table}[htbp]
   \centering
     \caption{\small\label{num2} Error in $L^\infty-$ and $L^1-$  norms for Example \ref{order2} at time $T=1$ without positivity preserving limiter.}

   \centering
    \begin{tabular}{|r|c|ccccc|}

    \hline
     $\mathcal{P}_{k}$ & $M$ & $\|u_n-u_h^n\|_{L^\infty(\Omega)}$ & Order& $\|u_n-u_h^n\|_{L^1(\Omega)}$ & Order & $\min u_h^n$  \\\hline
            & $40$ &7.33E-003& --&1.03E-003& --& -8.87e-005\\
       1   &$80$ &1.24e-003&2.56 &2.27e-004&2.18& -1.08e-005\\
            &$160$ &2.63e-004&2.24&5.44e-005&2.06& -4.41e-007\\
            &$320$ &6.05e-005&2.12&1.35e-005&2.01& -1.57e-008\\
    \hline
            &$40$ &1.70E-003& --&8.73E-005& -- & -1.60e-005\\
         2  &$80$ &1.43e-004&3.57 &8.07e-006&3.44& -1.79e-007\\
           &$160$ &1.36e-005&3.39&9.40e-007&3.10& -6.24e-009\\
            &$320$ &1.34e-006&3.34&1.16e-007&3.02& -2.07e-010\\

         \hline
          &$40$ &1.45e-004&-- &6.00e-006&--& -2.14e-006\\
       3   &$80$ &9.87e-006&3.88&3.11e-007&4.27& -9.56e-008\\
            &$160$ &5.51e-007&4.16&1.76e-008&4.14& -3.51e-009\\
             &$320$ &3.50e-008&3.98&1.11e-009&3.99& -1.19e-010\\

                  \hline
    \end{tabular}
    \bigskip
   \caption{\small\label{num3} Error in $L^\infty-$ and $L^1-$  norms for Example \ref{order2} at time $T=1$ with positivity  preserving limiter.}

   \centering
   \begin{tabular}{|r|c|ccccc|}
     \hline
    $\mathcal{P}_{k}$ & $M$ & $\|u_n-U_h^n\|_{L^\infty(\Omega)}$ & Order & $\|u_n-U_h^n\|_{L^1(\Omega)}$ & Order & $\min U_h^n$  \\\hline
           & $40$ &7.33E-003& --&1.05E-003& --& 2.05e-005\\
       1   &$80$ &1.24e-003&2.56 &2.27e-004&2.21& 8.15e-007\\
            &$160$ &2.63e-004&2.24&5.44e-005&2.06& 2.77e-008\\
            &$320$ &6.05e-005&2.12&1.35e-005&2.01& 8.55e-010\\
    \hline
             &$40$ &1.70E-003& --&8.73E-005& -- & 6.15e-008\\
         2  &$80$ &1.43e-004&3.57 &8.08e-006&3.43& 3.03e-007\\
           &$160$ &1.36e-005&3.39&9.40e-007&3.10& 1.08e-008\\
            &$320$ &1.34e-006&3.34&1.16e-007&3.02& 4.55e-010\\
         \hline
              &$40$ &1.45e-004&-- &6.02e-006&--& 1.00e-014\\
       3   &$80$ &9.87e-006&3.88&3.13e-007&4.27& 4.45e-008\\
            &$160$ &5.51e-007&4.16&1.77e-008&4.14& 1.21e-009\\
             &$320$ &3.50e-008&3.98&1.11e-009&4.00& 2.55e-011\\       \hline

    \end{tabular}

\end{table}
}

\end{ex}

 \subsection{Porous media equation}\label{diffusion}

For the porous media equation, $f(u)H''(u)$ can locally vanish, resulting in degenerate cases \cite{bessemoulin2012finite}. We test the asymptotic behavior of the numerical solution and will show that the KKT limiter is necessary. The entropy defined in (\ref{1.b}), which should be non-increasing, is also computed.


\begin{ex}\label{porous}
In order to test degenerate cases, we choose the following parameters in  (\ref{1.a}) on the domain $\Omega= [0,1]$ with zero-flux boundary conditions
 \begin{align*}
  f(u)=u,\quad H'(u)=\displaystyle\frac43\left(u-\frac12\right)^3\max\left(u,\frac12\right),\quad \phi(x)=0, \quad  x \in \Omega,
 \end{align*}
and initial data
 \begin{align*}
u(x,0)=\displaystyle\frac12-\frac12\cos(2 \pi x), \quad  x \in \Omega.
\end{align*}

During the computations, the value of $\alpha$ for optimal convergence of the semi-smooth Newton algorithm is usually close to 0.1. We present the numerical solution in Fig. \ref{num6} for basis functions with polynomial order ranging from 1 to 3 and with the KKT limiter enforced. Values of the Lagrange multiplier $\lambda$ larger than $10^{-10}$ are shown in Fig. \ref{num6}, which indicate that the positivity constraint works well since it is only active at locations where the solution is close to the minimum value. The entropy decay using the KKT limiter and polynomial basis functions of order 3 is presented in Fig. \ref{num71}, in which the result is consistent with the stability analysis. In Fig. \ref{num72}, the numerical solution without KKT limiter and for polynomial basis functions with order 3 is plotted. This computation breaks down due to unphysical oscillations.
 \begin{figure*}[htbp]
\small
\centering
\subfigure [$\mathcal{P}_{1}$]{\includegraphics[height=5.5cm,width=5.1cm]{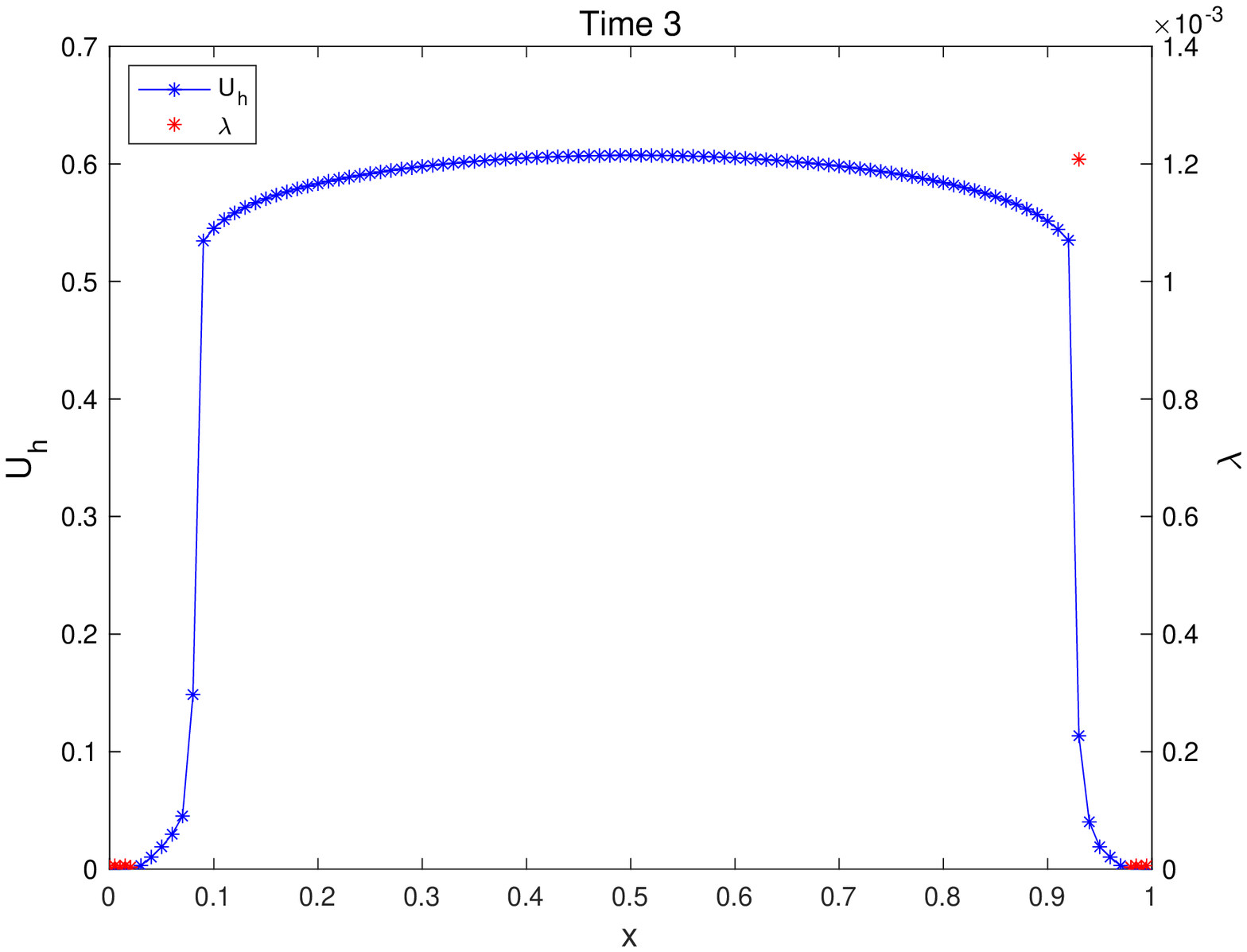}}
\subfigure [$\mathcal{P}_{2}$]{\includegraphics[height=5.5cm,width=5.1cm]{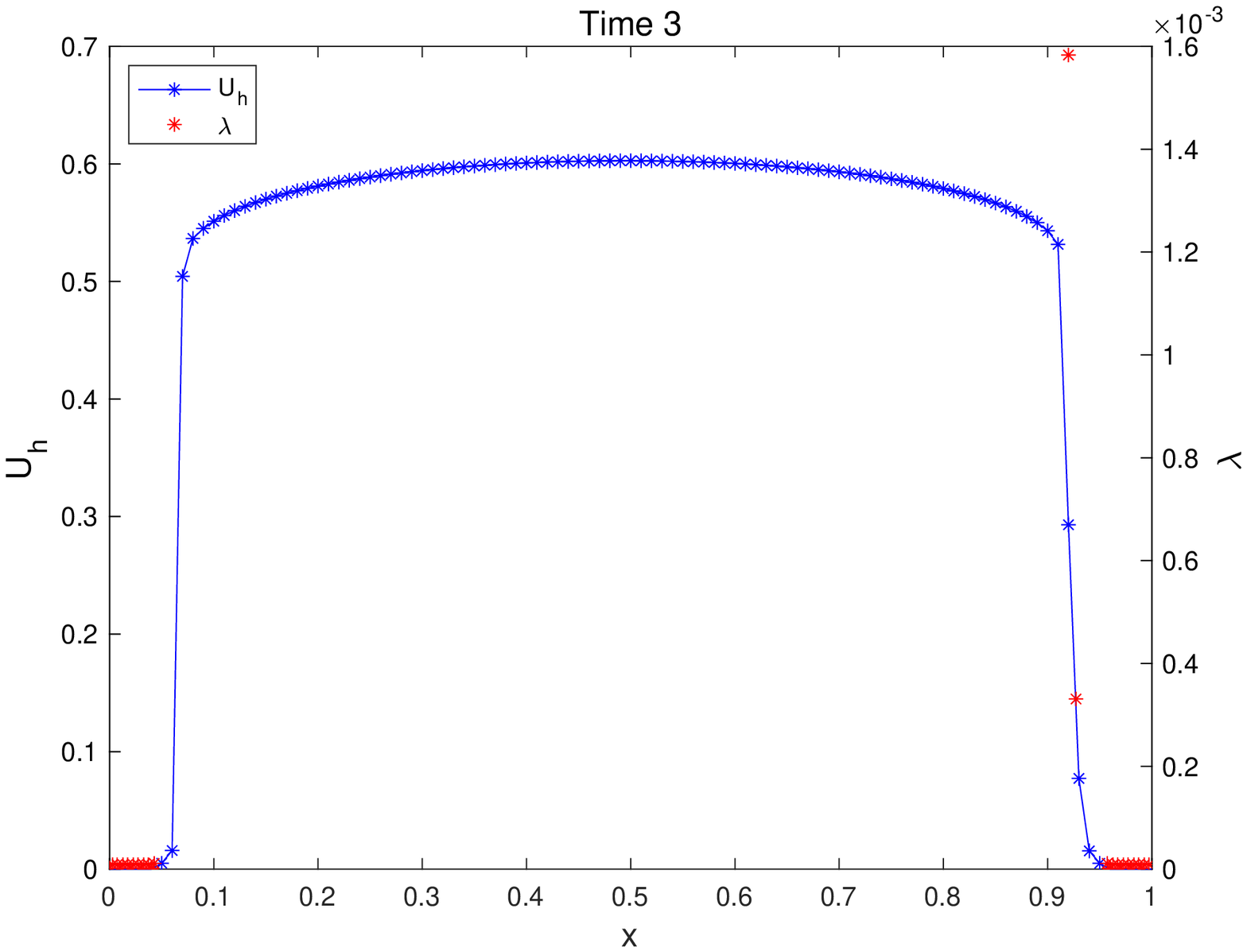}}
\subfigure[$\mathcal{P}_{3}$] {\includegraphics[height=5.5cm,width=5.1cm]{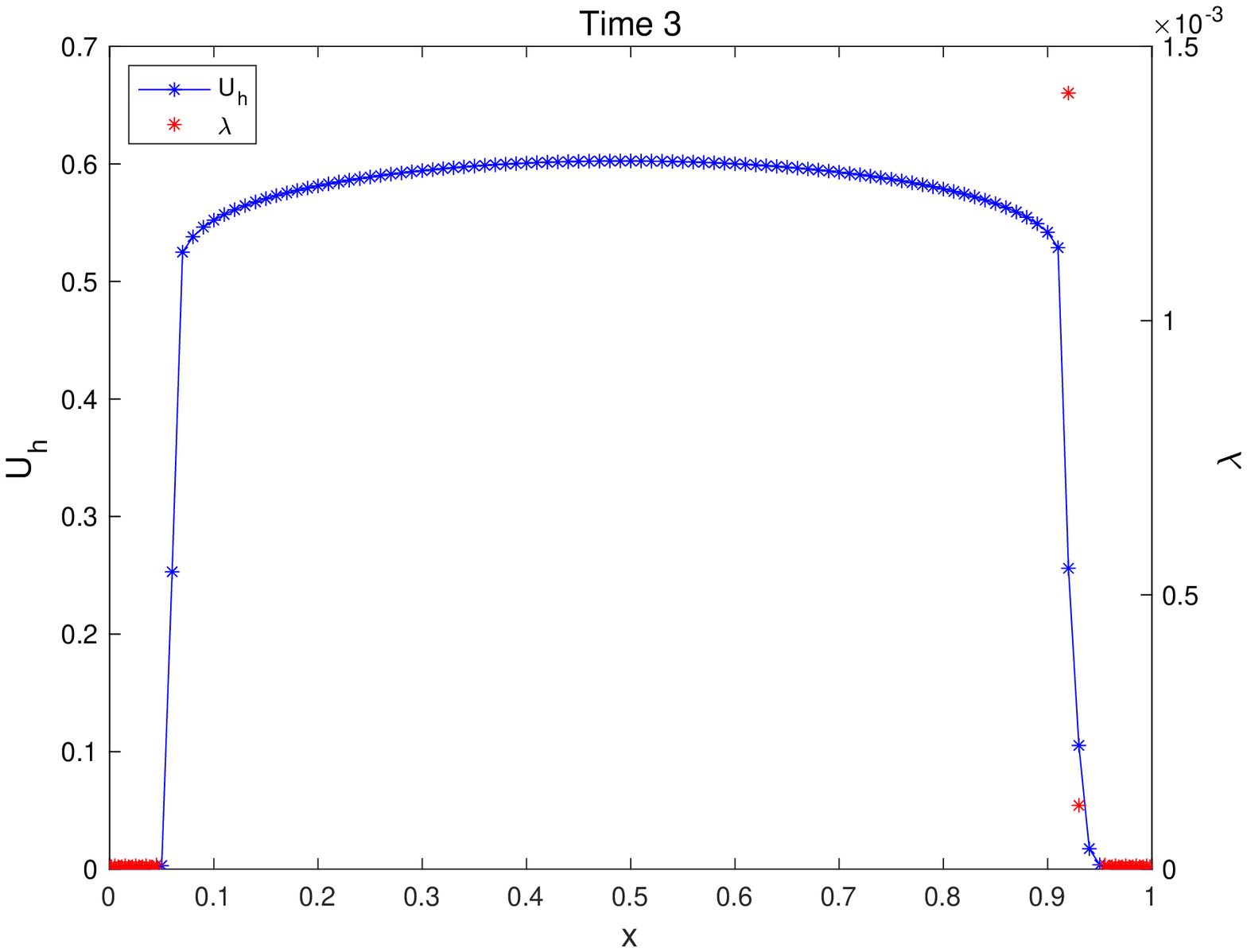}}
 \caption{\small\label{num6}(Example \ref{porous}) Numerical solution $U_h$ for different orders of polynomial basis functions $\mathcal{P}_{1}$-$\mathcal{P}_{3}$ with the KKT limiter enforced and Lagrange multiplier $\lambda$ (red dots).}
 \end{figure*}

 \begin{figure*}[htbp]
\small
\centering
 \subfigure {\includegraphics[height=7cm,width=7.9cm]{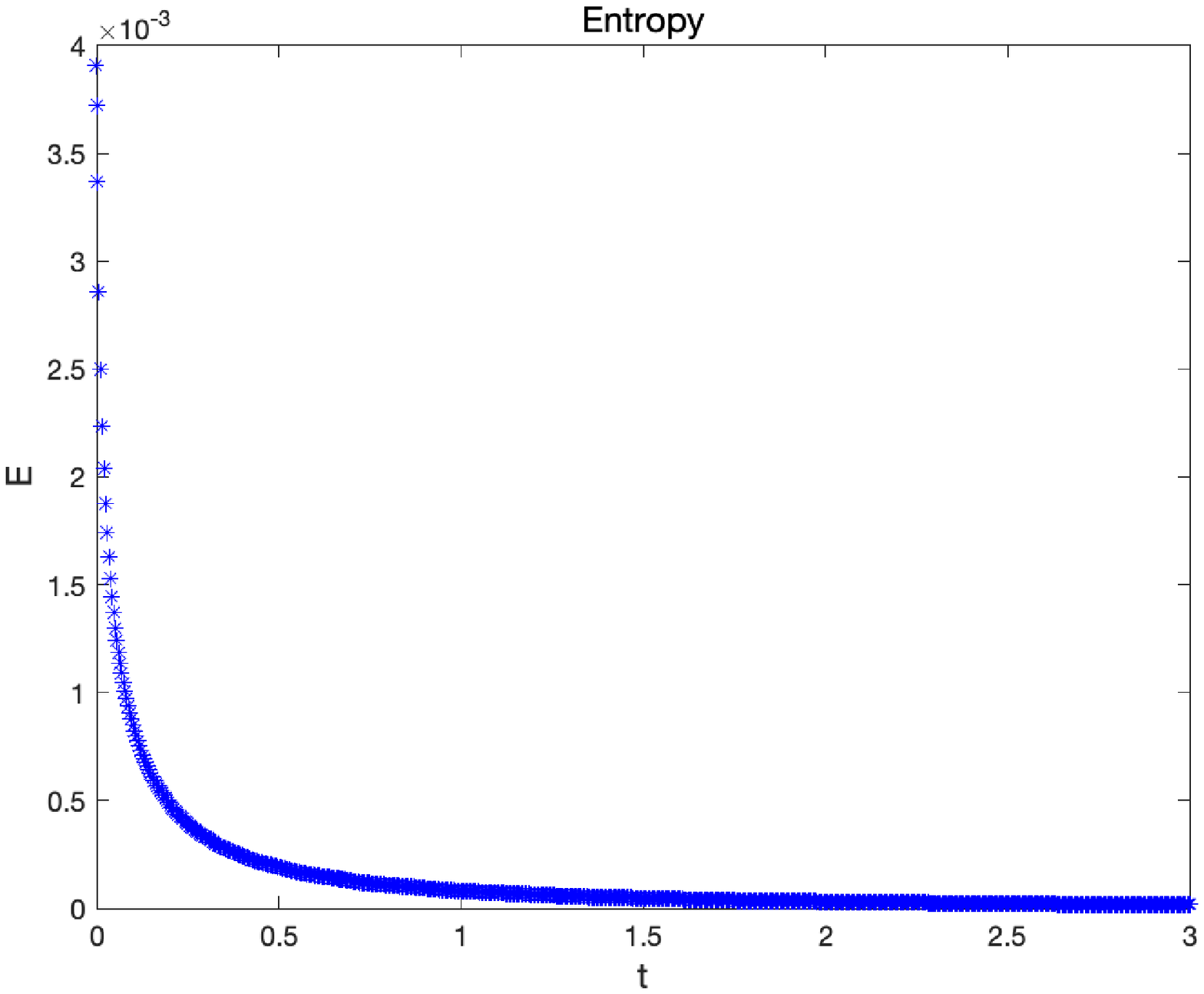}}
 \caption{\small\label{num71}(Example \ref{porous}) Entropy $E_h$ for $\mathcal{P}_{3}$ basis functions with the KKT limiter enforced.}
\end{figure*}

 \begin{figure*}[htbp]
\small
\centering
\subfigure {\includegraphics[height=7cm,width=7.9cm]{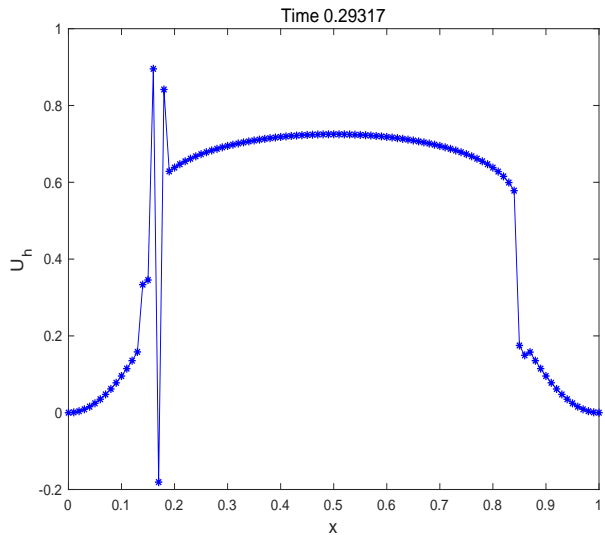}}
 \caption{\small\label{num72}(Example \ref{porous}) Numerical solution $U_h$ for $\mathcal{P}_{3}$ basis functions \emph{without} KKT limiter just before blow up.}
\end{figure*}
\end{ex}

 \begin{ex}\label{porous3}
 We consider a 2D test case on the domain $\Omega=[-6,6]^2$ with zero-flux boundary conditions by choosing in  (\ref{1.a}) the following parameters
 \begin{align*}
 f(u)=u,\quad H'(u)=2u,\quad \phi(\pmb{x})=0,\quad \pmb{x}\in \Omega,
 \end{align*}
and initial data
 \begin{align*}
u(\pmb{x},0)=\exp\left(-\frac{1}{2}|\pmb{x}|^2\right),\quad \pmb{x}\in \Omega.
 \end{align*}

The value of $\alpha$ in the definition of the time step ranges in this case between 0.1 and 1. Fig. \ref{num10} presents the numerical solution with the KKT limiter active and also the Lagrange multiplier $\lambda$. Considering the position of the non-zero Lagrange multipliers, we can see that the limiter also works well in the two-dimensional case since it is only active in areas where positivity must be enforced. The entropy decay is plotted in Fig. \ref{num11}, which is consistent with the stability result of the numerical solution. Without the KKT limiter, there will be unphysical oscillations, and the computation will break down at some point in the computations.

\begin{figure*}[htbp]
\small
\centering
\subfigure{\includegraphics[height=6.5cm,width=7.5cm]{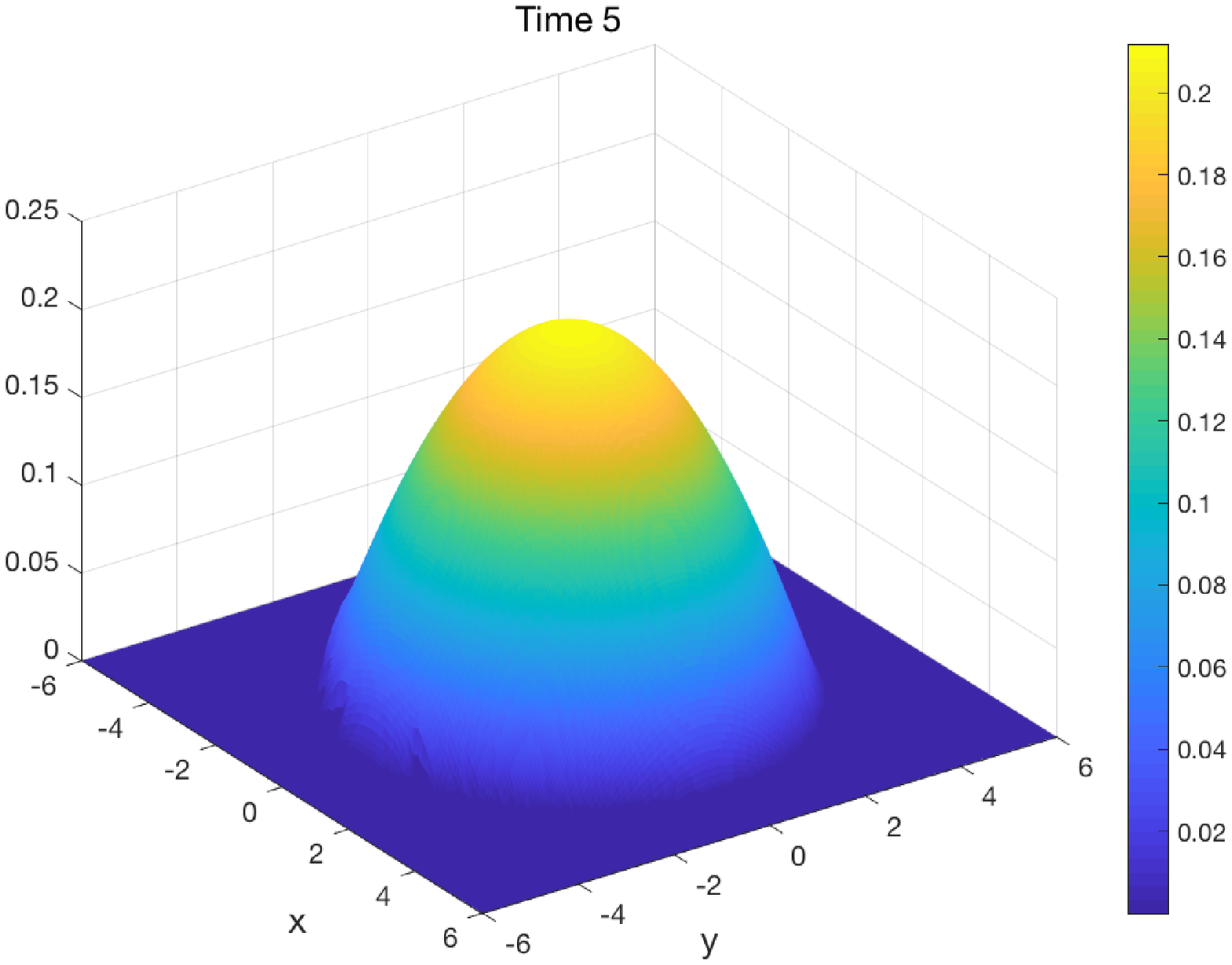}}
\subfigure{\includegraphics[height=6.5cm,width=7.5cm]{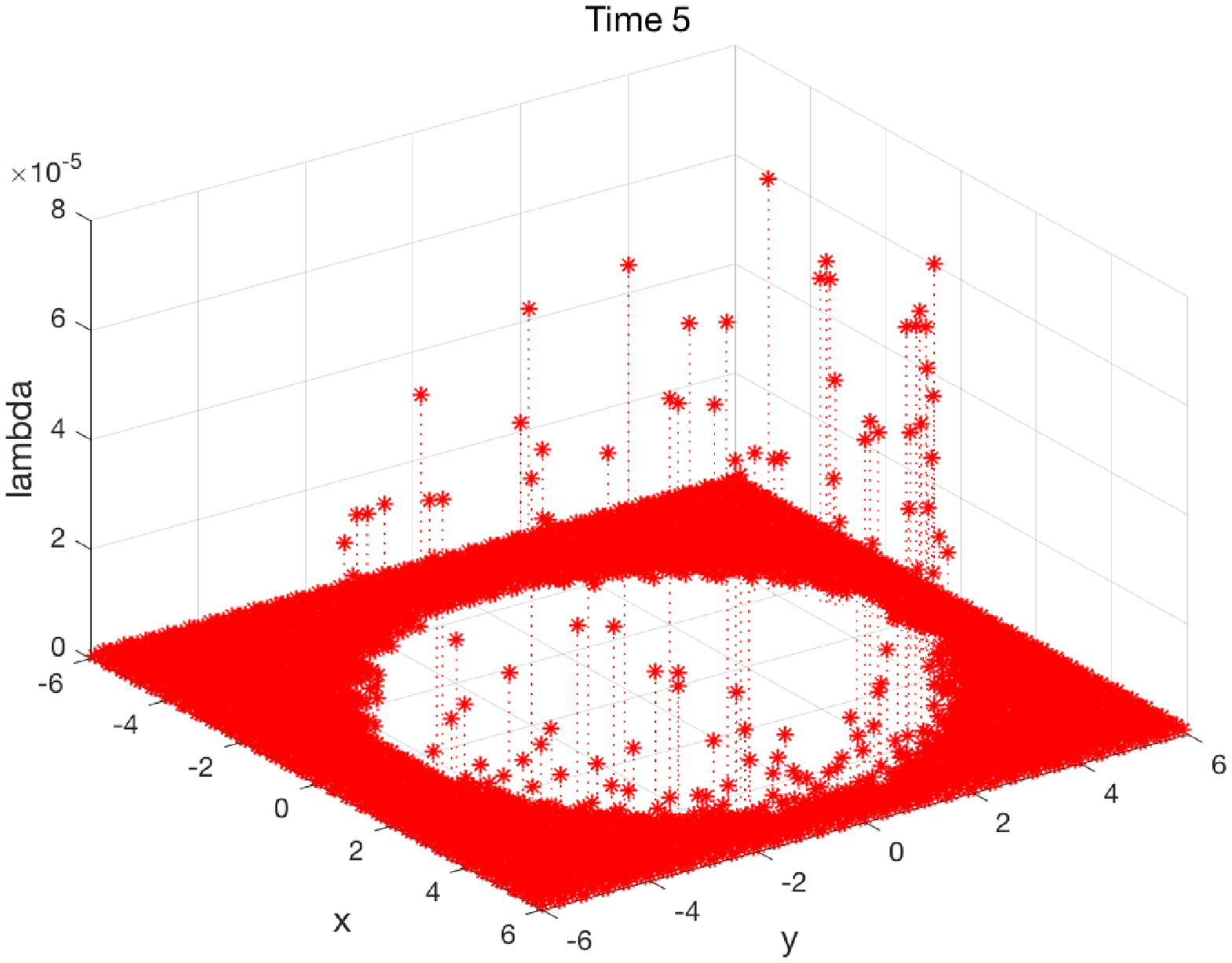}}
 \caption{\small\label{num10}(Example \ref{porous3}) Numerical solution $U_h$ for $\mathcal{P}_{2}$ basis functions with KKT limiter enforced (Left) and Lagrange multiplier $\lambda$ (Right).}
\subfigure{\includegraphics[height=6.05cm,width=7.05cm]{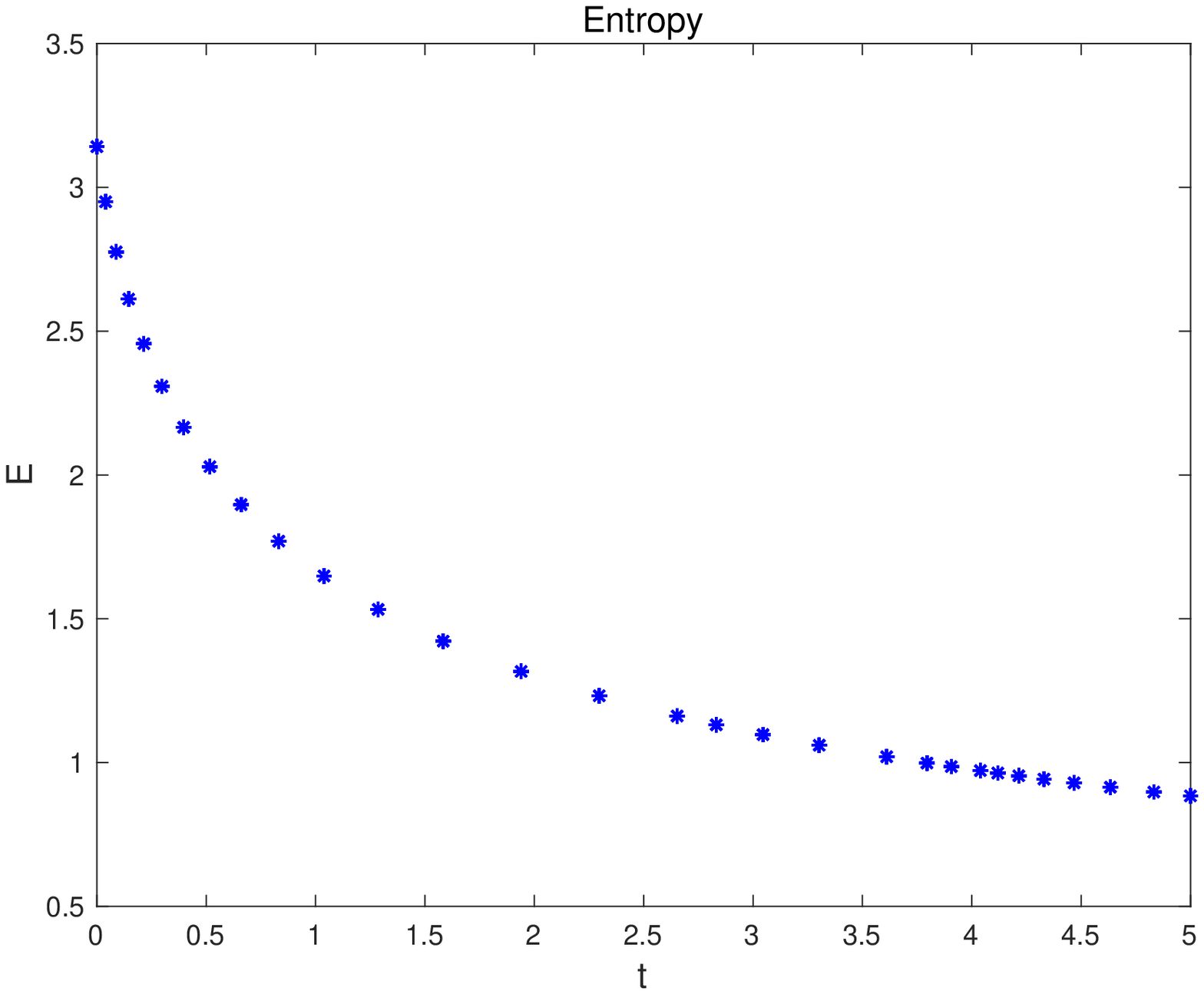}}
 \caption{\small\label{num11}(Example \ref{porous3}) Entropy $E_h$ for $\mathcal{P}_{2}$ basis functions with KKT limiter enforced.}
\end{figure*}
\end{ex}

  \subsection{Nonlinear diffusion with a double-well potential}\label{double-well}

 Consider the nonlinear diffusion equation with double-well potential \cite{kalmykov2007brownian} on the domain $\Omega=[-1.4,1.4]$, which is obtained by choosing in (\ref{1.a}) zero-flux boundary conditions and the following parameters
 \begin{align}\label{**}
 f(u)=u,\quad H'(u)=u ,\quad \phi(x)=\frac{1}{4}x^4-\frac{1}{2}x^2, \quad x\in \Omega.
 \end{align}
This model is taken from \cite{carrillo2015finite}. We will test the evolution of the numerical solution with and without KKT limiter, and also the decay of the entropy (\ref{1.b}). The value of $\alpha$ to compute the time step ranges between 0.01 to 0.1.

\begin{ex}\label{double-well1}
We consider (\ref{1.a}) with (\ref{**})  and the initial data
 \begin{align*}
u(x,0)=\frac{0.2}{\sqrt{0.4\pi}}\exp\left({-\frac{x^2}{0.4}}\right),\quad x\in \Omega.
 \end{align*}

The numerical solution with the KKT limiter enforced and the values of the Lagrange multiplier $\lambda$ larger than $10^{-10}$ are shown in Fig. \ref{num4}. These results indicate that the numerical solution tends to a steady state and that the KKT limiter is only active at places where the positivity constraint needs to be imposed. The entropy dissipation is presented in Fig. \ref{num5}, in which uniform decay coincides with our theoretical analysis. For the numerical solution without the KKT limiter, we observe that violating the positivity constraint will result in discontinuities in the solution and a computation breakdown, even for a very small CFL number.

 \begin{figure*}[htbp]
\small
\centering
\subfigure{\includegraphics[height=6.25cm,width=6.25cm]{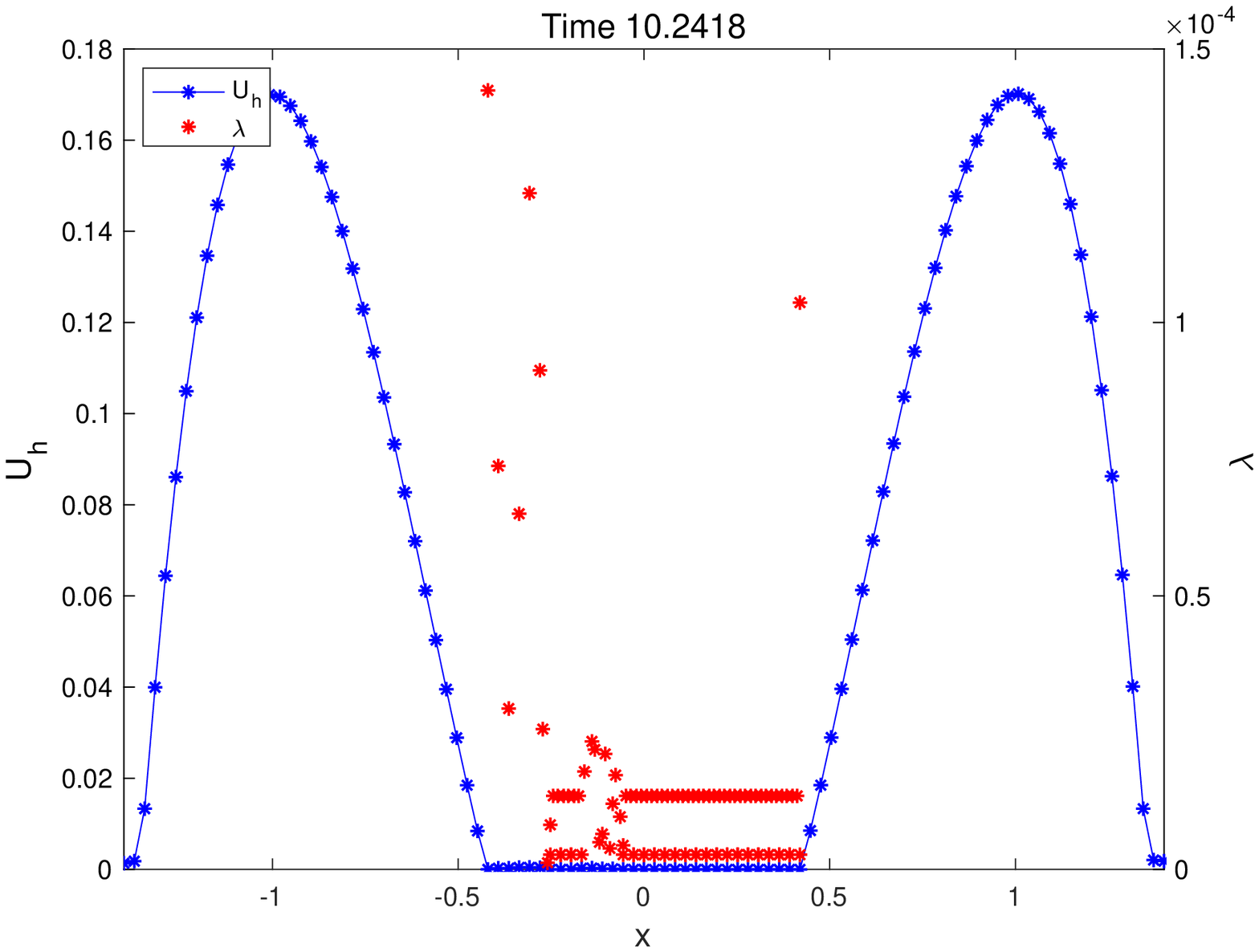}}
\subfigure{\includegraphics[height=6.25cm,width=6.25cm]{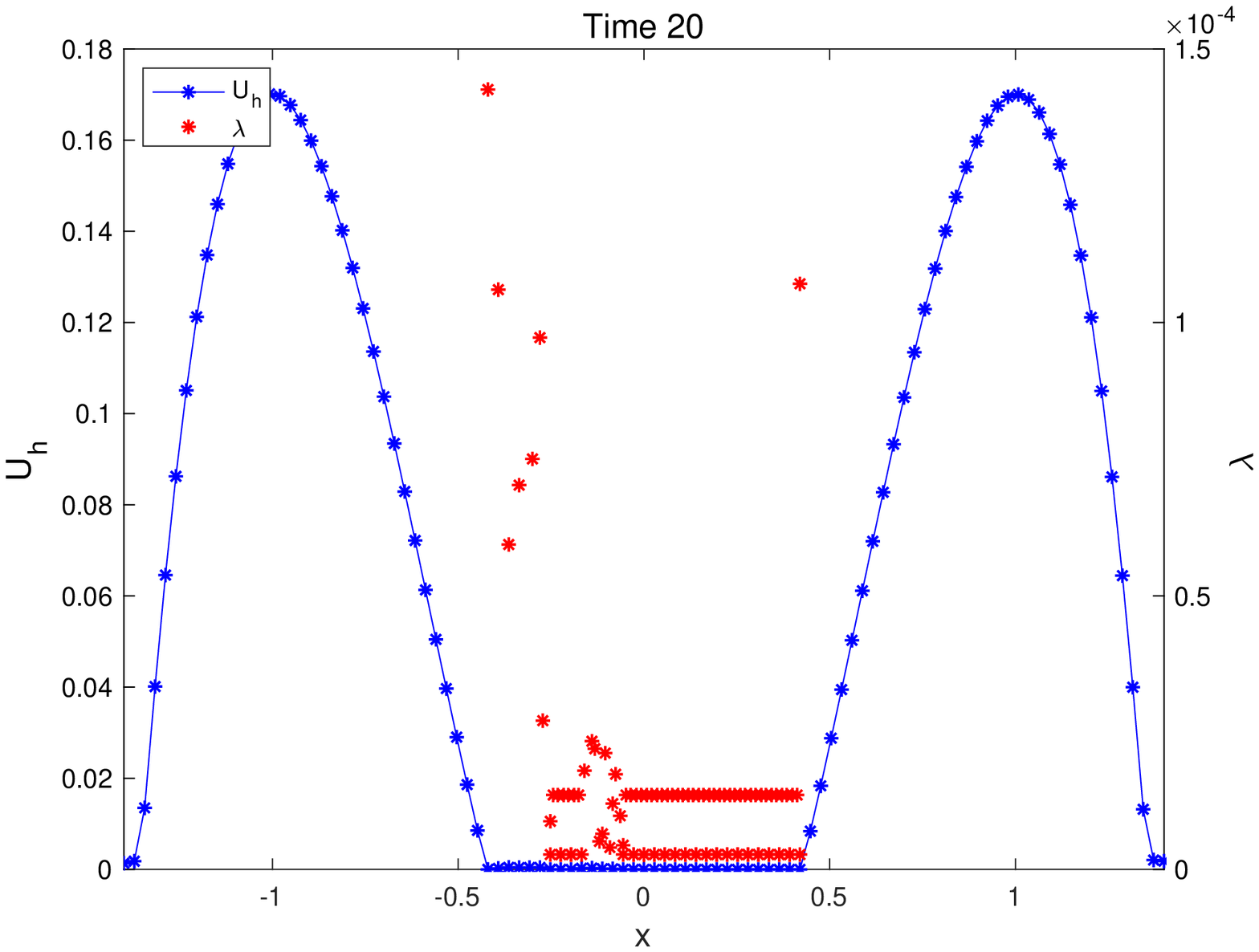}}
 \caption{\small\label{num4}(Example \ref{double-well1}) Numerical solution $U_h$ for $\mathcal{P}_{2}$ basis functions with KKT limiter enforced and Lagrange multiplier $\lambda$ (red dots).}
\small
\centering
\subfigure{\includegraphics[height=6.6cm,width=8cm]{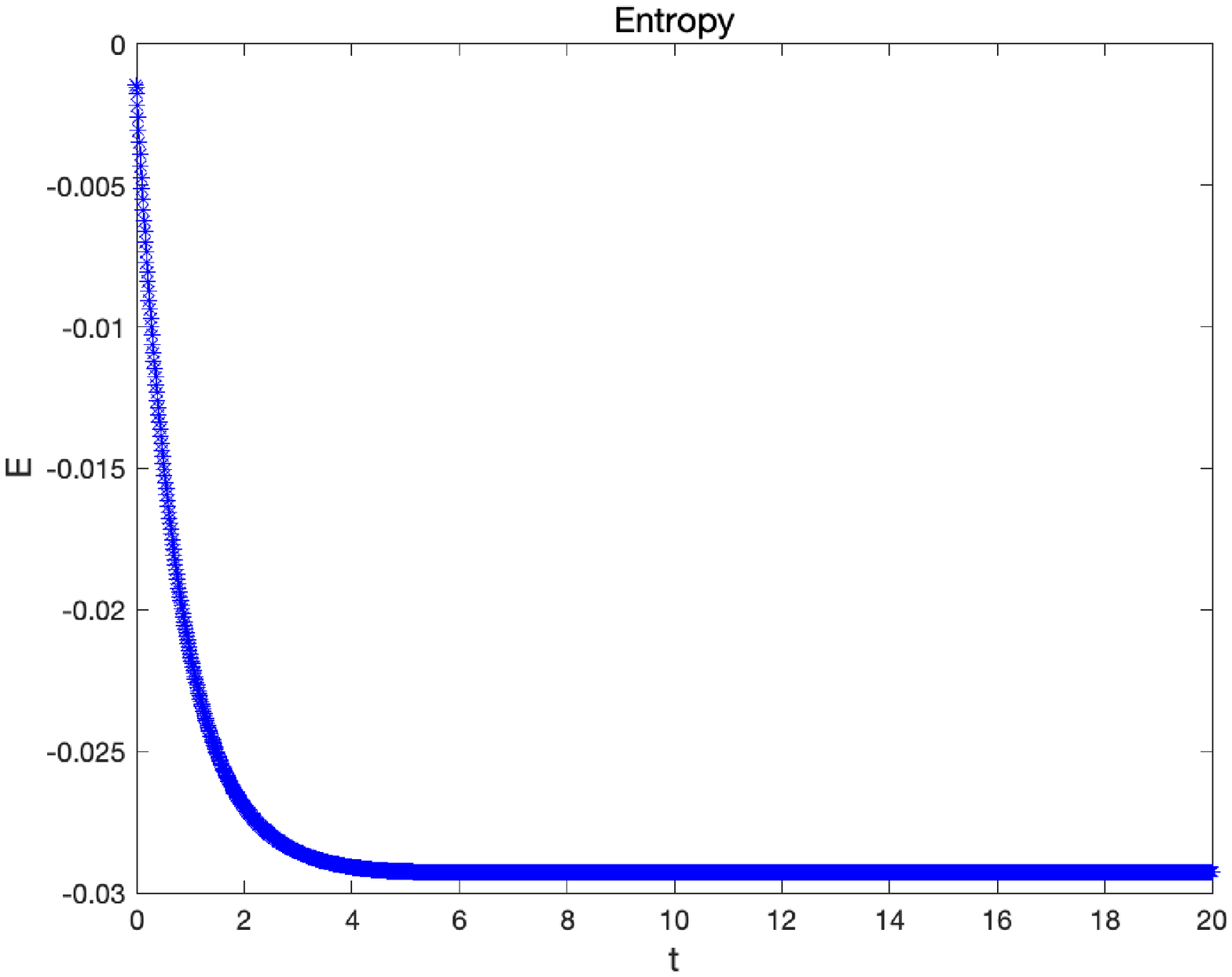}}
 \caption{\small\label{num5}(Example \ref{double-well1}) Entropy $E_h$ for $\mathcal{P}_{2}$ basis functions with KKT limiter enforced.}
\end{figure*}
\end{ex}

 \subsection {Nonlinear Fokker-Plank equation for fermion gases}\label{Fermion1}
 We consider the nonlinear Fokker-Plank equation for fermion gases \cite{bessemoulin2012finite} on the domain $\Omega=[-10,10]^2$, for which we select the following parameters in (\ref{1.a})
  \begin{align}\label{546}
 f(u)=u(1-u),\quad H'(u)=\log \frac{u}{1-u},\quad \phi(\pmb{x})=\frac{1}{2}|\pmb{x}|^2,\quad \pmb{x}\in\Omega,
 \end{align}
together with zero-flux boundary conditions.

 \begin{ex}\label{Fermion2D}

 We consider (\ref{1.a}) with (\ref{546}) and initial data
 \begin{align*}
\displaystyle u(\pmb{x},0)=&\frac{1}{2\sqrt{2\pi}}\left(\exp\left(-\frac{1}{2}|\pmb{x}-(2,2)|^2\right)+\exp\left(-\frac{1}{2}|\pmb{x}-(2,-2)|^2\right)\right.\\
&+\left.\exp\left(-\frac{1}{2}|\pmb{x}-(-2,2)|^2\right)+\exp\left(-\frac{1}{2}|\pmb{x}-(-2,-2)|^2\right)\right),\quad \pmb{x}\in\Omega.
 \end{align*}

During the computations, the value of $\alpha$ in the definition of the time step ranges between 0.1 and 1, but for most time steps $\alpha=1$. The numerical solutions at several time levels with the KKT limiter enforced and the entropy dissipation are presented in Figs \ref{num13} and \ref{num13-1}, respectively, showing the time-asymptotic convergence of the numerical solution towards a steady state. Without the KKT limiter, the computations break down, even for very small CFL numbers.

\begin{figure*}[htbp]
\small
\centering
\subfigure {\includegraphics[height=6.25cm,width=6.25cm]{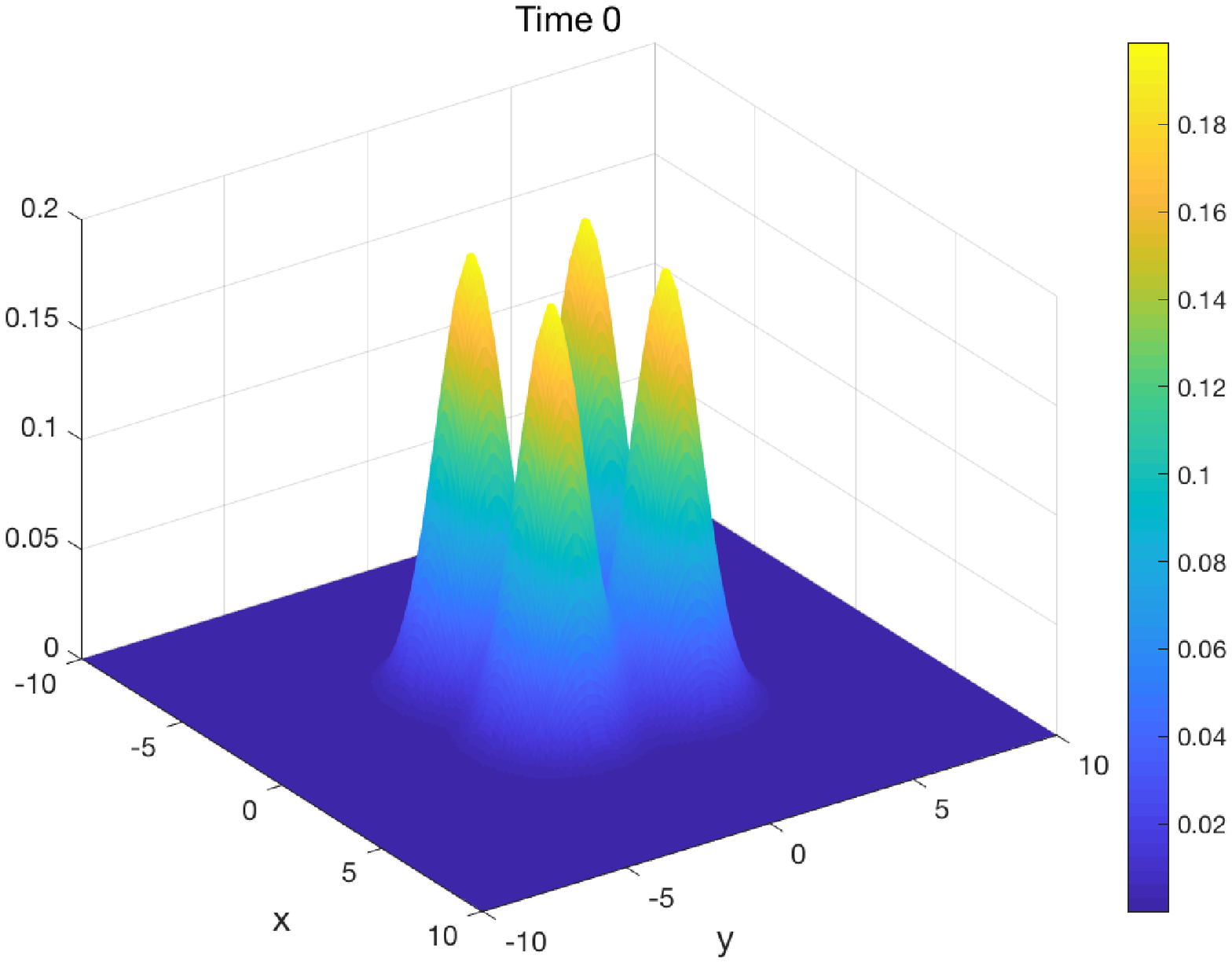}}
\subfigure{\includegraphics[height=6.25cm,width=6.25cm]{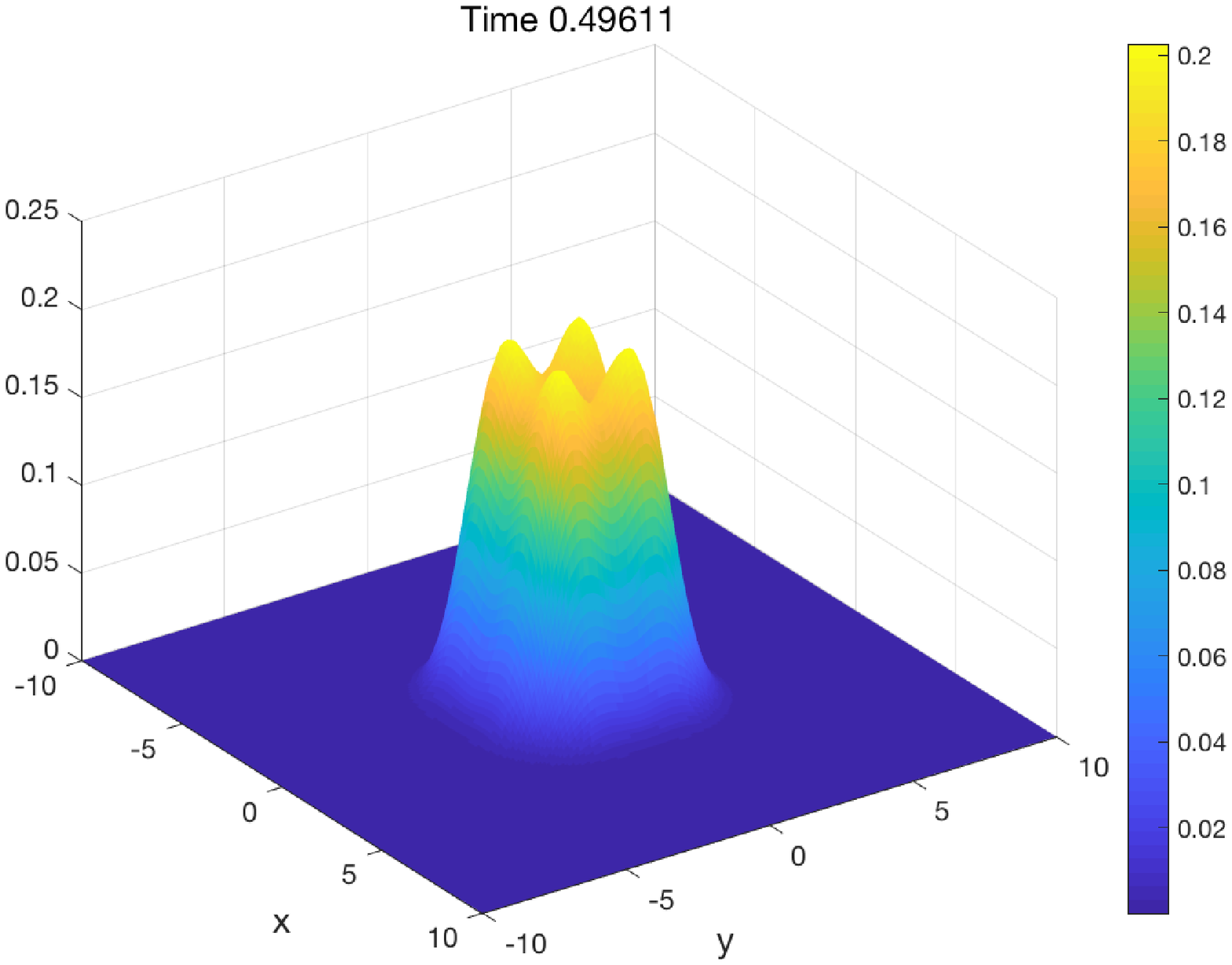}}
\subfigure{\includegraphics[height=6.25cm,width=6.25cm]{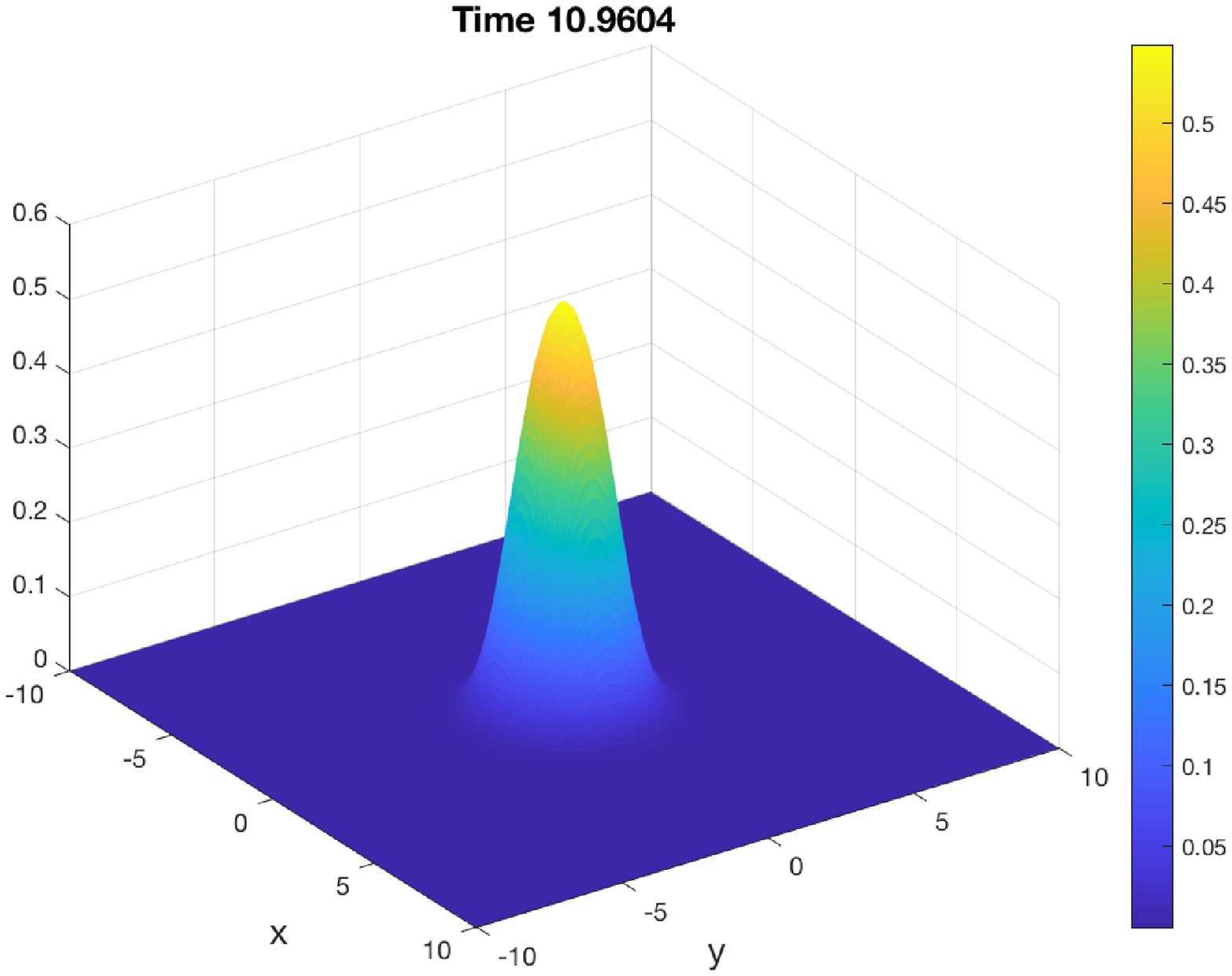}}
\subfigure{\includegraphics[height=6.25cm,width=6.25cm]{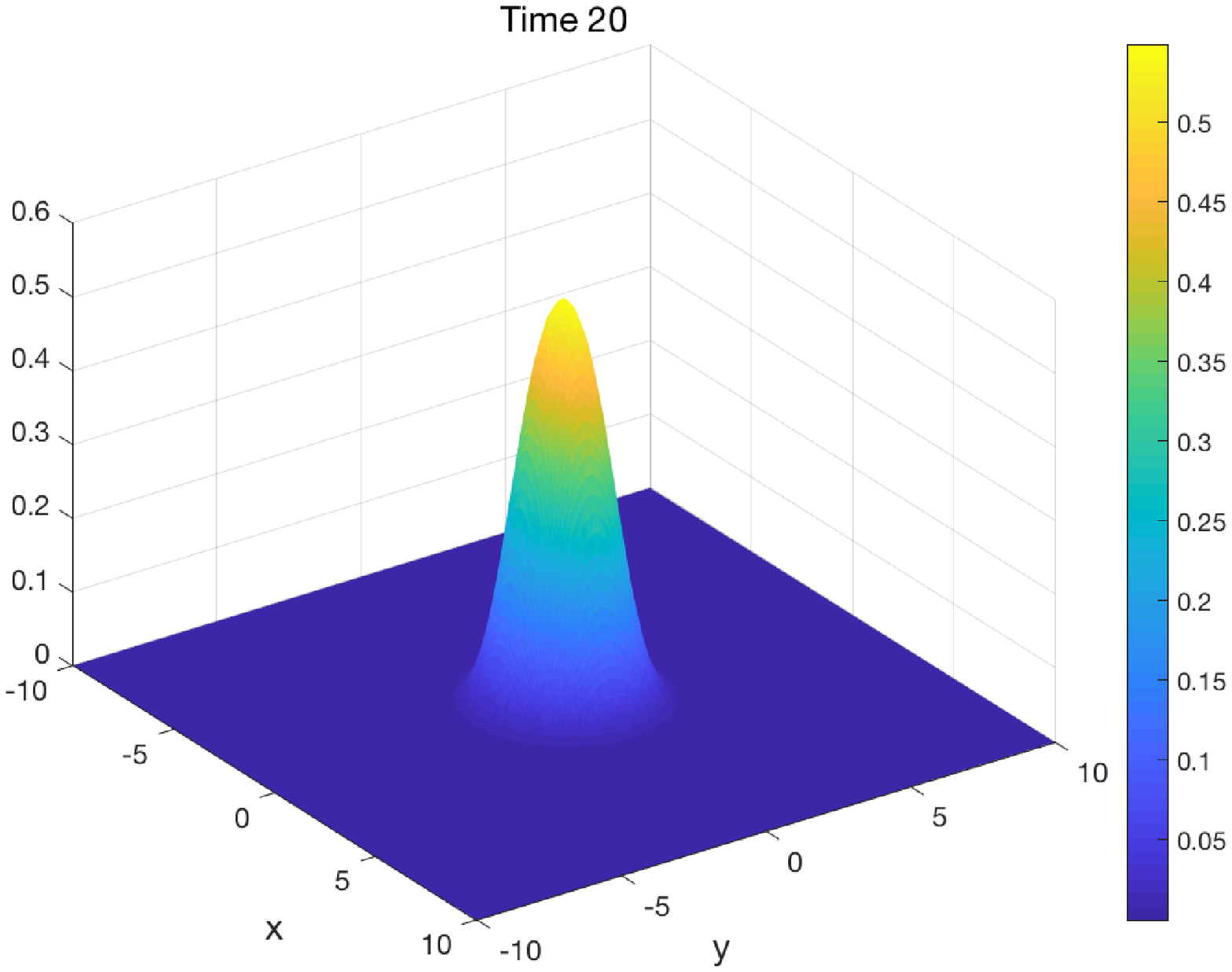}}
 \caption{\small\label{num13}(Example \ref{Fermion2D}) Numerical solution $U_h$ for $\mathcal{P}_{2}$ basis functions with KKT limiter enforced.}
\end{figure*}
\begin{figure*}[htbp]
\small
\centering
\subfigure{\includegraphics[height=6.3cm,width=6.5cm]{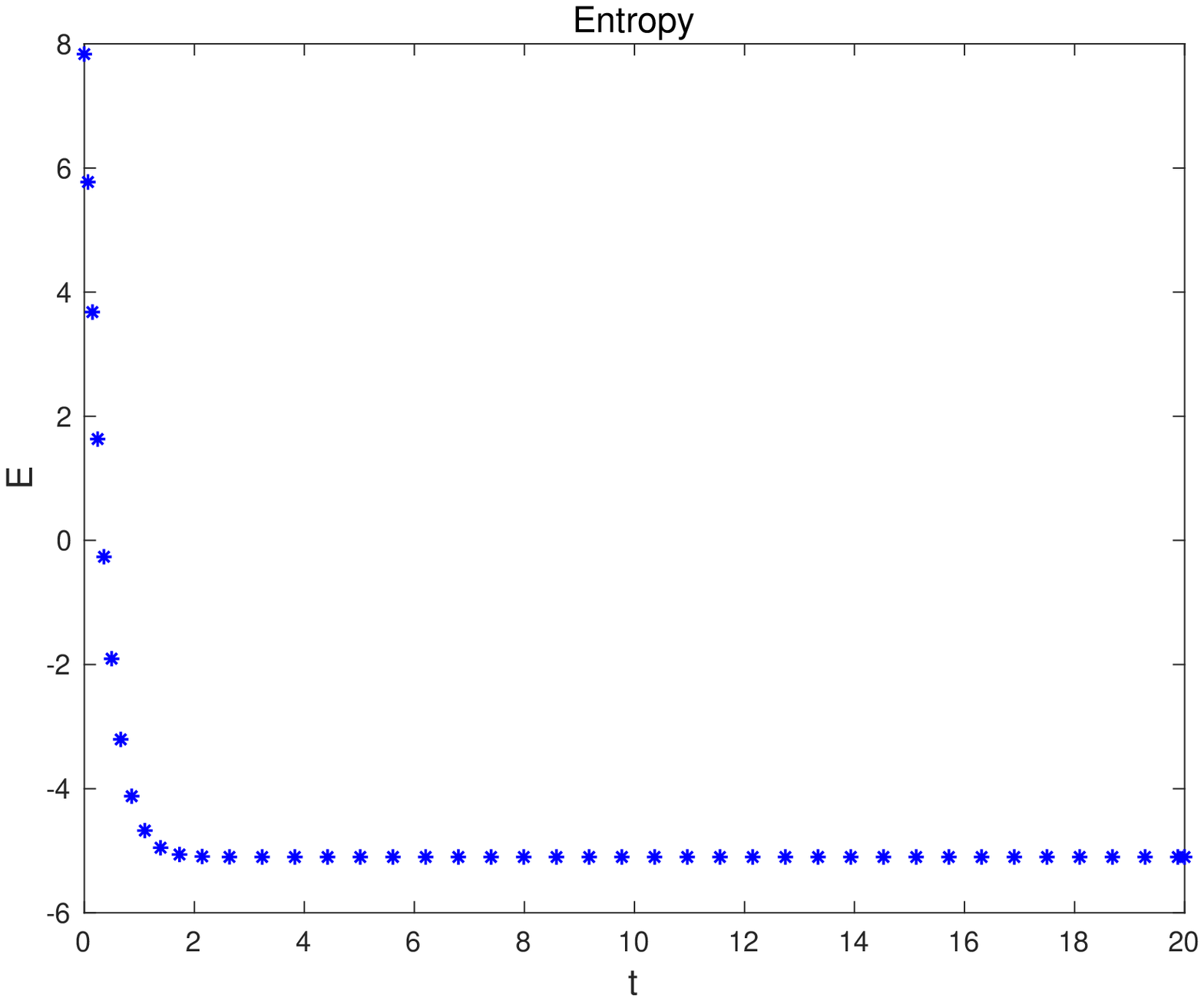}}
 \caption{\small\label{num13-1}(Example \ref{Fermion2D}) Entropy $E_h$ for $\mathcal{P}_{2}$ basis functions with KKT limiter enforced.}
\end{figure*}
\end{ex}

\subsection{Nonlinear Fokker-Plank equation for boson gases}\label{boson}
\begin{ex}\label{boson1}
We consider a nonlinear Fokker-Plank equation for boson gases  with zero-flux boundary condition on a domain $\Omega=[-10,10]$, which requires the following parameters in (\ref{1.a})
 \begin{align*}
 f(u)=u(1+u^3),\quad H'(u)=\log \frac{u}{(1+u^3)^{\frac13}} ,\quad \phi(x)=\frac{x^2}{2},\quad x\in\Omega.
 \end{align*}
The initial data is \cite{bessemoulin2012finite, liu2016entropy}
 \begin{align*}
\displaystyle u(x,0)=&\frac{M}{2\sqrt{2\pi}}\left(\exp\left({-\frac{(x-2)^2}{2}}\right)+\exp\left({-\frac{(x+2)^2}{2}}\right)\right),\quad x\in \Omega,
 \end{align*}
 where $M\geqslant 0$ is the mass of $u(x,0)$.

For most time steps, the value of $\alpha$  in the definition of the time step is 1. For the case $M=1$, Fig. \ref{num14-1} displays the numerical solution at various times. Also, the locations and values of the Lagrange multiplier $\lambda$ and the entropy with the KKT limiter enforced are shown. The results in Figs \ref{num14-1} and \ref{num14} indicate that the numerical solution tends to a steady state, and that the Lagrange multiplier $\lambda$ is needed to ensure that the positivity constraint is satisfied. Without the KKT limiter, the computations break down, even for very small CFL numbers.

For this model equation, there is a critical mass phenomenon \cite{abdallah2011minimization}, which states that solutions with a large initial mass blows-up in a finite time, while solutions with a small mass at an initial time will not. The numerical solutions with sub-critical mass $M=1$ at times $t=5$ and $t=10$ and with super-critical mass $M=10$ at times $t=0.2$ and $t=1$ are shown in Fig. \ref{num15} and Fig. \ref{num16}, respectively, and agree with the results shown in \cite{abdallah2011minimization} and the numerical observation in \cite{bessemoulin2012finite, liu2016entropy}.

\begin{figure*}[htbp]
\small
\centering
\subfigure{\includegraphics[height=6.25cm,width=6.25cm]{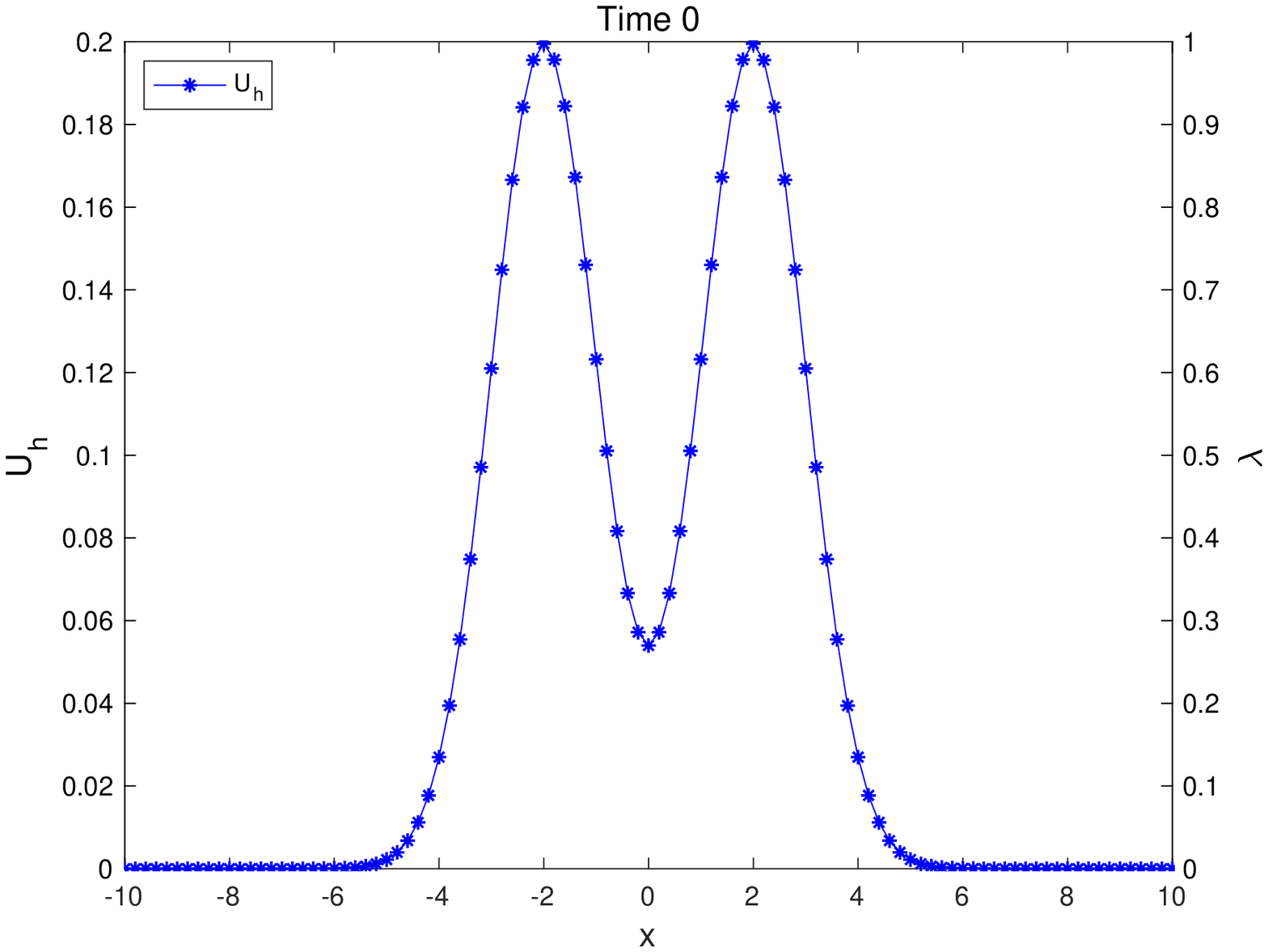}}
\subfigure{\includegraphics[height=6.25cm,width=6.25cm]{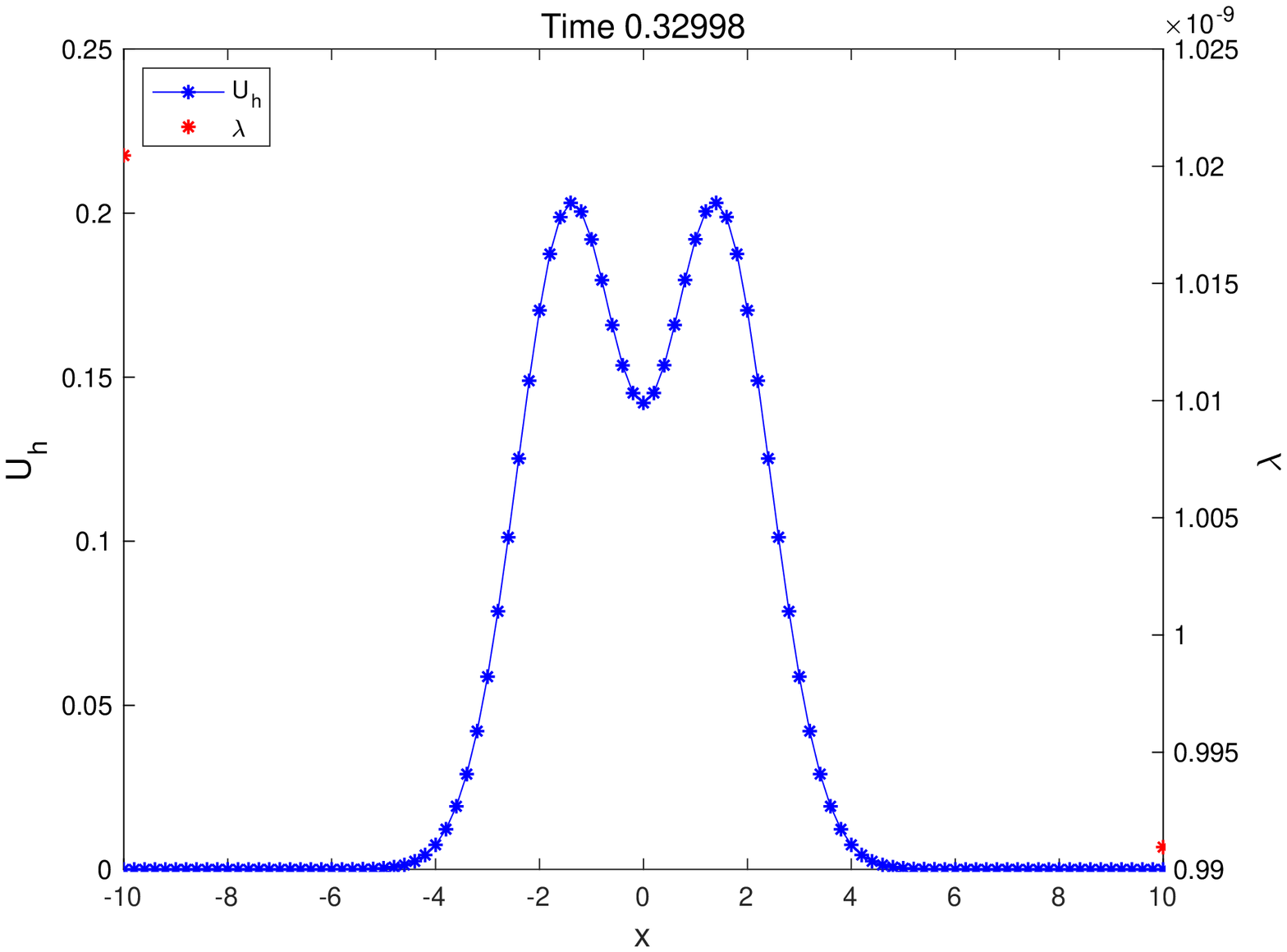}}
\subfigure{\includegraphics[height=6.25cm,width=6.25cm]{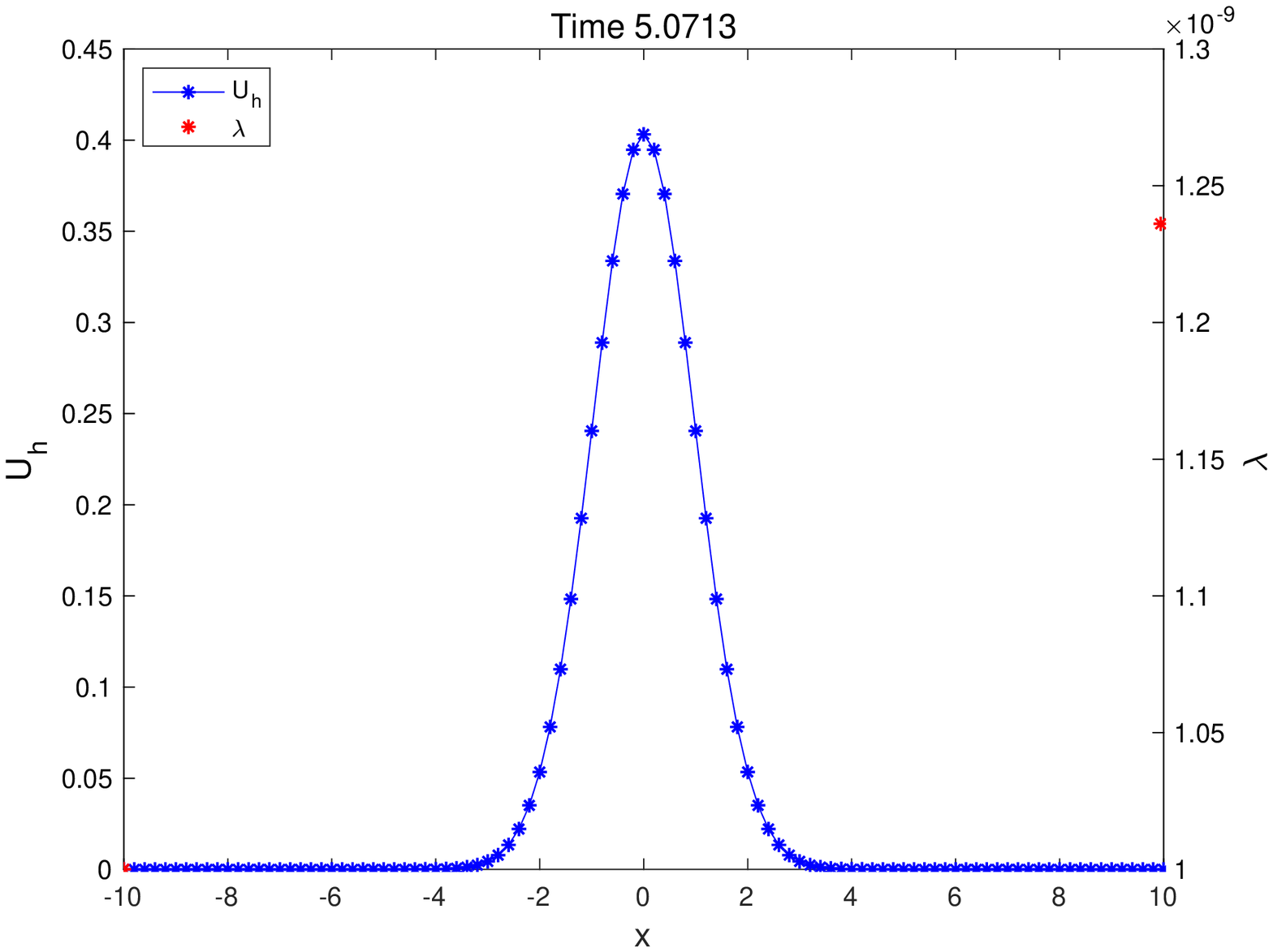}}
\subfigure{\includegraphics[height=6.25cm,width=6.25cm]{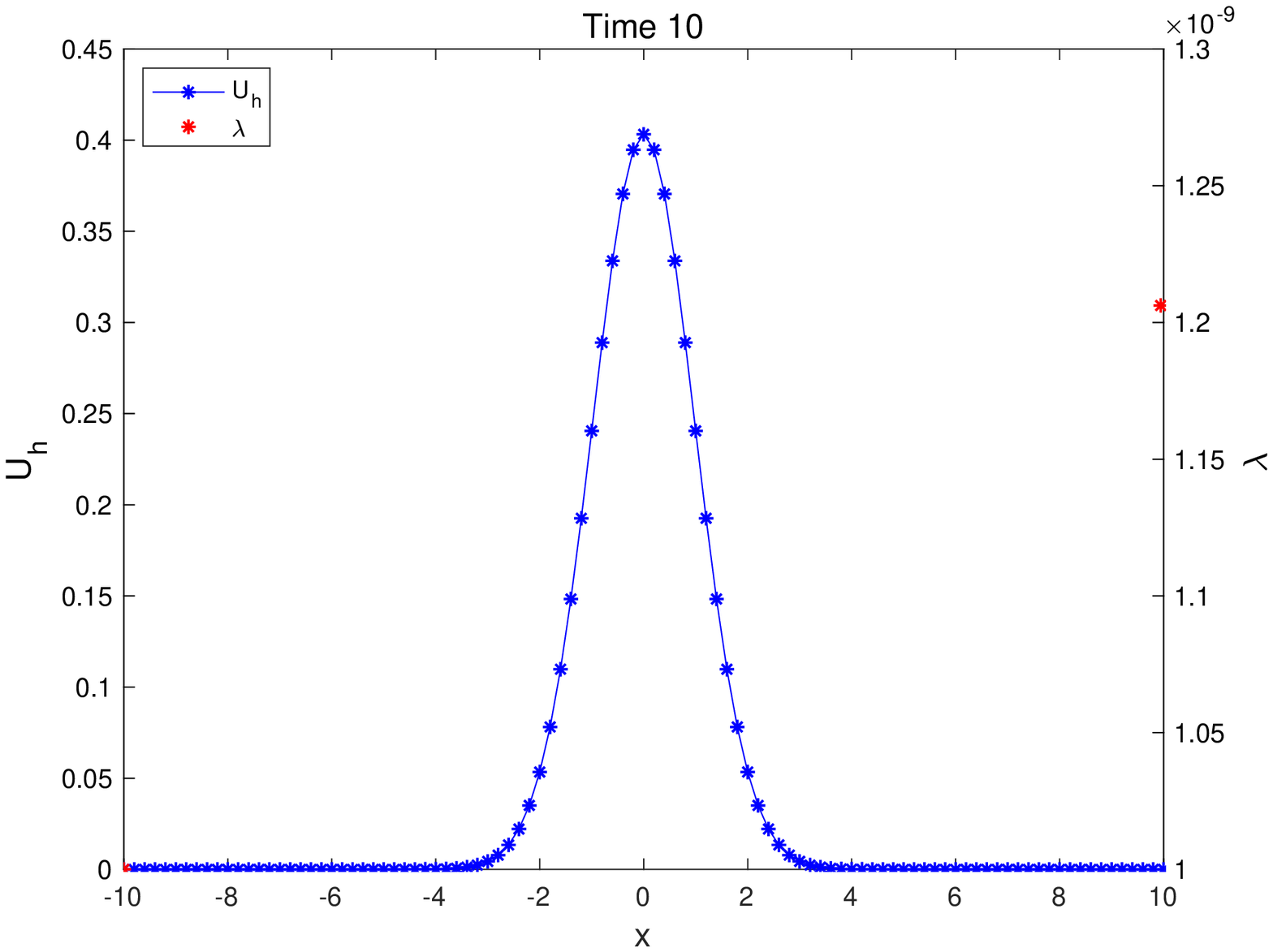}}
 \caption{\small\label{num14-1}(Example \ref{boson1}): Numerical solution $U_h$ for $\mathcal{P}_{2}$ basis functions with KKT limiter enforced.}
\end{figure*}

\begin{figure*}[htbp]
\small
\centering
\subfigure {\includegraphics[height=6.5cm,width=7.5cm]{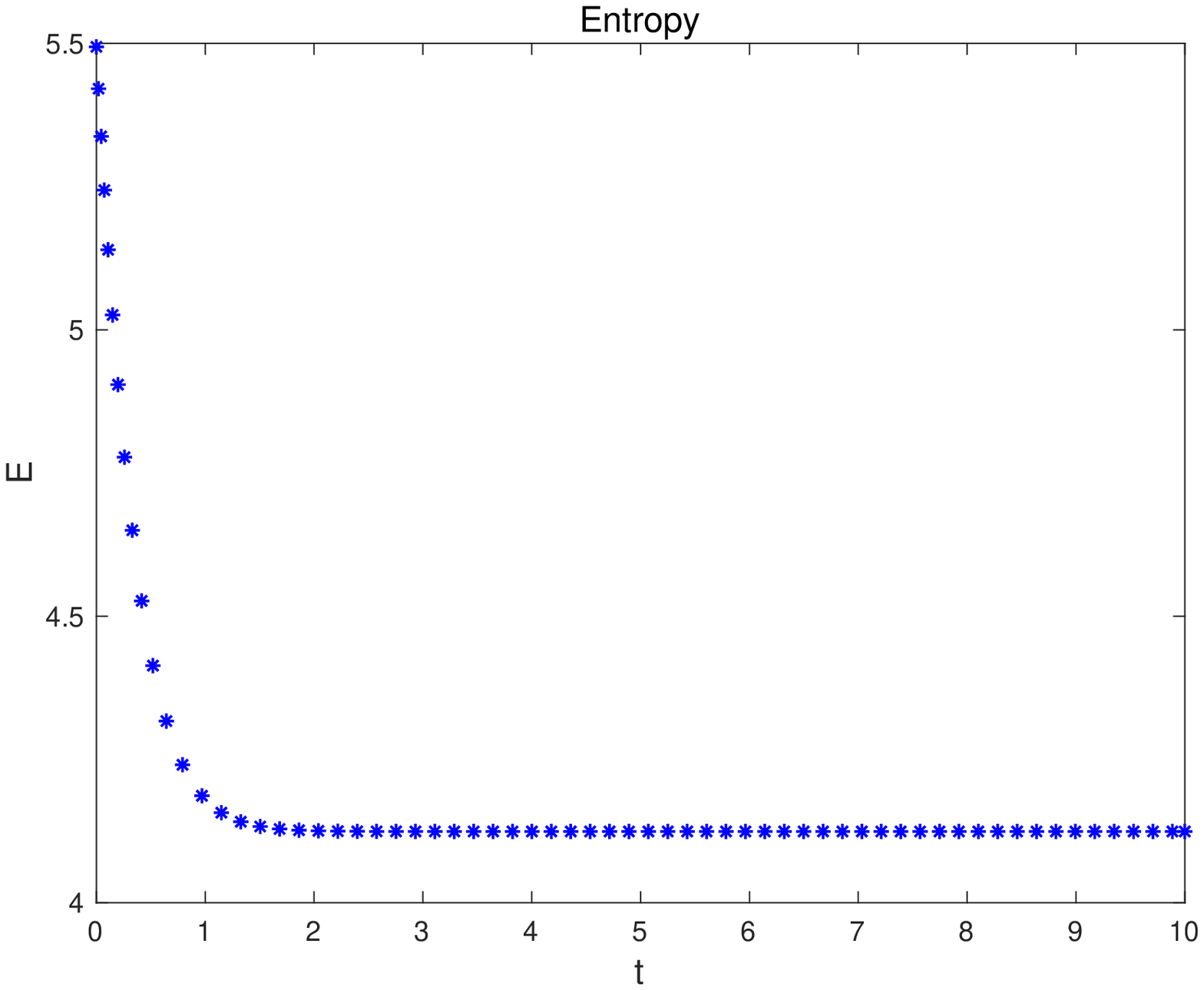}}
 \caption{\small\label{num14}(Example \ref{boson1}): Entropy $E_h$ for $\mathcal{P}_{2}$ basis functions with KKT limiter enforced.}
\end{figure*}

\begin{figure*}[htbp]
\small
\centering
\subfigure{\includegraphics[height=6.25cm,width=6.25cm]{Boson-gas-5.eps}}
\subfigure{\includegraphics[height=6.25cm,width=6.25cm]{Boson-gas-10.eps}}
 \caption{\small\label{num15}(Example \ref{boson1}: $M=1$): Numerical solution $U_h$ for $\mathcal{P}_{2}$ basis functions with KKT limiter enforced.}
\subfigure{\includegraphics[height=6.25cm,width=6.25cm]{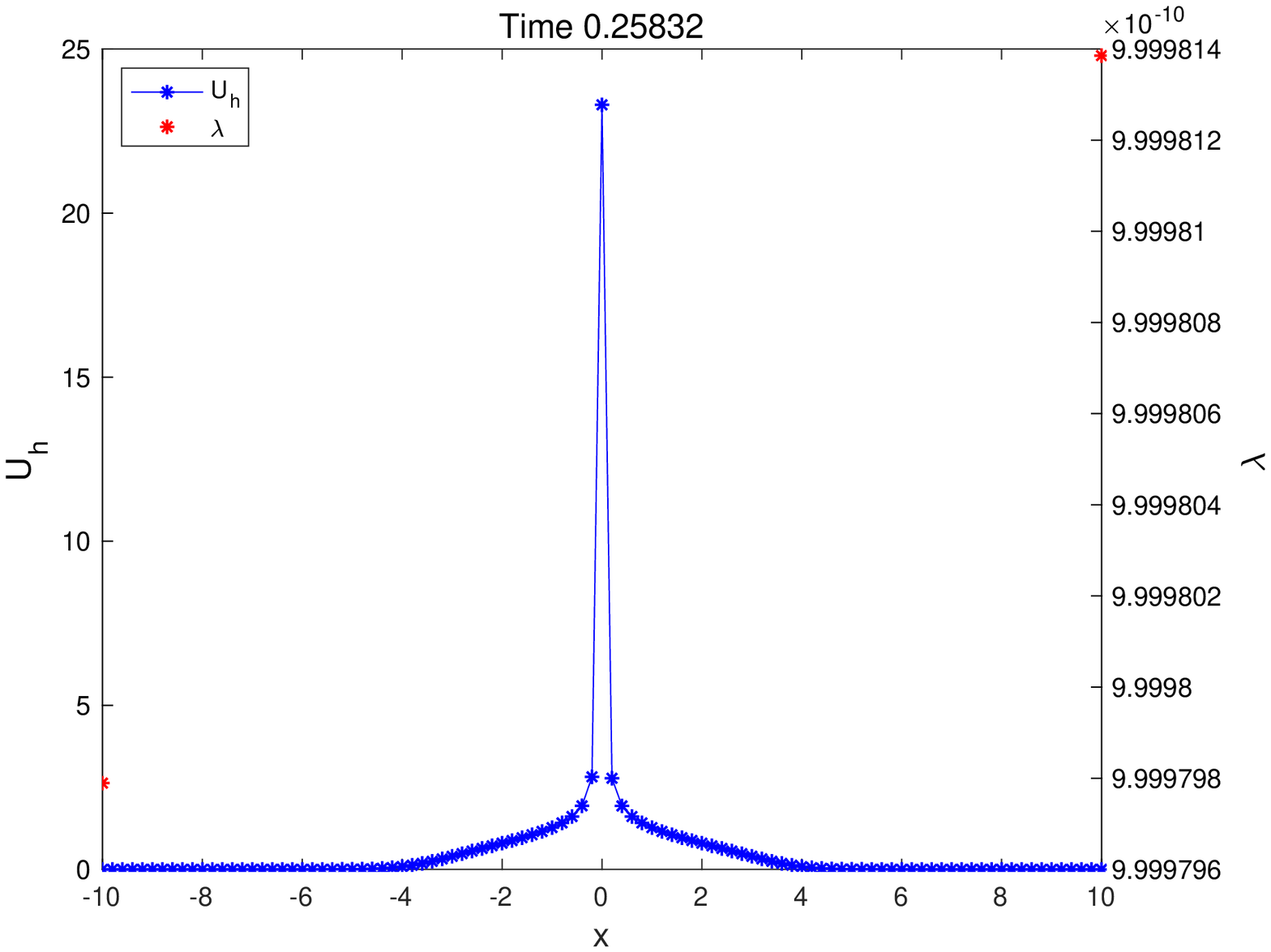}}
\subfigure {\includegraphics[height=6.25cm,width=6.25cm]{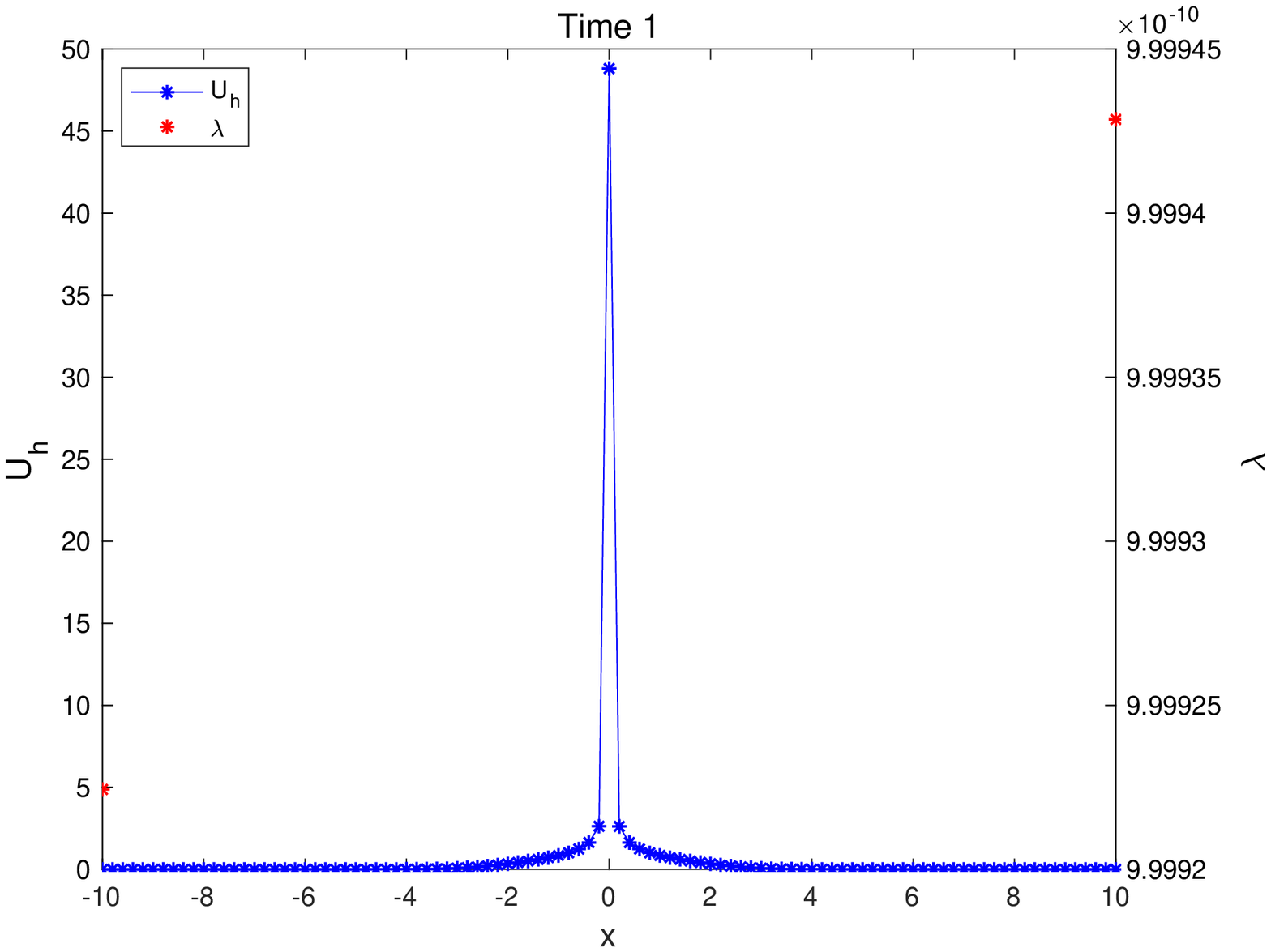}}
 \caption{\small\label{num16}(Example \ref{boson1}: $M=10$) Numerical solution $U_h$ for $\mathcal{P}_{2}$ basis functions with KKT limiter enforced.}
\end{figure*}
\end{ex}

\section{Conclusions}
\label{sec:conclusions}

The main topic of this paper is the formulation of higher order accurate positivity preserving DIRK-LDG discretizations for the nonlinear degenerate parabolic equation (\ref{1.a}). The presented numerical discretizations allow the combination of a positivity preserving limiter and time-implicit numerical discretizations for PDEs and alleviate the time step restrictions of currently available positivity preserving DG discretizations, which generally require the use of explicit time integration methods. For the spatial discretization an LDG method combined with a simple alternating numerical flux is used, which simplifies the theoretical analysis for the entropy dissipation. For the temporal discretization, the implicit DIRK methods significantly enlarge the time-step required for stability of the numerical discretization. We prove the existence, uniqueness and unconditional entropy dissipation of the positivity preserving high order accurate KKT-LDG discretization combined with an implicit Euler time discretization. Numerical results are presented to demonstrate the accuracy of the higher order accurate positivity preserving KKT-DIRK-LDG discretizations, which are of optimal order and not affected by the positivity preserving KKT limiter. The numerical solutions satisfy the entropy decay condition.

\section*{Acknowledgement}
\label{sec:acknowledgements}
The research of Fengna Yan was funded by a fellowship from the China Scholarship Council (No. 201806340058). The research of J.J.W. van der Vegt was partially supported by the University 
of Science and Technology of China (USTC), Hefei, Anhui, China, while the author was in residence at USTC.
The research of Yinhua Xia  was partially supported by National Natural Science Foundation of China grant No. 12271498. The research of Yan Xu was partially supported by National Natural Science Foundation of China grant No. 12071455.

\end{document}